\documentclass[reqno,11pt,letterpaper]{amsart}

\usepackage{amsmath}
\usepackage{amssymb}
\usepackage{amsthm}
\usepackage{mathrsfs}
\usepackage{accents}
\usepackage{calc}
\usepackage{arydshln}
\usepackage{upgreek}
\usepackage{slashed}
\usepackage{xifthen}
\usepackage{graphicx}
\usepackage{longtable}
\usepackage[inline]{enumitem}

\usepackage{xcolor}
\definecolor{winered}{rgb}{0.6,0,0}
\definecolor{lessblue}{rgb}{0,0,0.7}

\usepackage[pdftex,colorlinks=true,linkcolor=winered,citecolor=lessblue,urlcolor=lessblue,breaklinks=true,bookmarksopen=true]{hyperref}

\makeatletter
\newcommand{\myitem}[3]{\item[#2]\def\@currentlabel{#3}\label{#1}}
\makeatother

\usepackage{titletoc}

\makeatletter

\def\@tocline#1#2#3#4#5#6#7{
\begingroup
  \par
    \parindent\z@ \leftskip#3 \relax \advance\leftskip\@tempdima\relax
                  \rightskip\@pnumwidth plus 4em \parfillskip-\@pnumwidth
    \ifcase #1 
       \vskip 0.6em \hskip 0em 
       \or
       \or \hskip 0em 
       \or \hskip 1em 
    \fi%
    #6
    \nobreak\relax{\leavevmode\leaders\hbox{\,.}\hfill}
    \hbox to\@pnumwidth {\@tocpagenum{#7}}
  \par
\endgroup
}

 \def\l@section{\@tocline{0}{0pt}{0pc}{}{}}

\renewcommand{\tocsection}[3]{%
  \indentlabel{\@ifnotempty{#2}{ 
    \ignorespaces\bfseries{#2. #3}}}
  \indentlabel{\@ifempty{#2}{\ignorespaces\bfseries{#3}}{}} 
    \vspace{1.5pt}}

\renewcommand{\tocsubsection}[3]{%
  \indentlabel{\@ifnotempty{#2}{
    \ignorespaces#2. #3}}
  \indentlabel{\@ifempty{#2}{\ignorespaces #3}{}}
    \vspace{1.5pt}}

\renewcommand{\tocsubsubsection}[3]{%
  \indentlabel{\@ifnotempty{#2}{
    \ignorespaces#2. #3}}
  \indentlabel{\@ifempty{#2}{\ignorespaces #3}{}}
    \vspace{1.5pt}}

\makeatother

\makeatletter
\def\@nomenstarted{0}
\newlength{\@nomenoldtabcolsep}

\newcommand{\nomenstart}
  {%
    \def\@nomenstarted{1}%
    \setlength{\@nomenoldtabcolsep}{\tabcolsep}%
    \setlength{\tabcolsep}{3.5pt}%
    \begin{longtable}{p{0.11\textwidth} p{0.86\textwidth}}
  }

\newcommand{\nomenitem}[2]{%
    \ifcase\@nomenstarted%
      \or 
      \or \\ 
    \fi%
    #1\,{\leavevmode\leaders\hbox{\,.}\hfill} & #2%
    \def\@nomenstarted{2}%
  }%
\newcommand{\nomenend}
  {\\%
      \end{longtable}%
      \setlength{\tabcolsep}{\@nomenoldtabcolsep}%
      \def\@nomenstarted{0}%
  }
\makeatother

\makeatletter
\newcommand{\vast}{\bBigg@{4}}
\newcommand{\Vast}{\bBigg@{5}}
\newcommand{\VAST}[1]{\bBigg@{#1}}
\makeatother

\allowdisplaybreaks

\numberwithin{equation}{section}
\numberwithin{figure}{section}
\newtheorem{thm}{Theorem}[section]

\newtheorem{prop}[thm]{Proposition}
\newtheorem{lemma}[thm]{Lemma}
\newtheorem{cor}[thm]{Corollary}
\newtheorem*{thm*}{Theorem}
\newtheorem*{prop*}{Proposition}
\newtheorem*{cor*}{Corollary}
\newtheorem*{conj*}{Conjecture}

\theoremstyle{definition}
\newtheorem{definition}[thm]{Definition}

\theoremstyle{remark}
\newtheorem{rmk}[thm]{Remark}

\makeatletter
\newcommand{\fakephantomsection}{%
  \Hy@MakeCurrentHref{\@currenvir.\the\Hy@linkcounter}
  \Hy@raisedlink{\hyper@anchorstart{\@currentHref}\hyper@anchorend}%
  \Hy@GlobalStepCount\Hy@linkcounter%
}
\makeatother

\newcommand{\mc}{\mathcal}
\newcommand{\cA}{\mc A}

\newcommand{\cC}{\mc C}

\newcommand{\cE}{\mc E}
\newcommand{\cF}{\mc F}
\newcommand{\cG}{\mc G}
\newcommand{\cH}{\mc H}
\newcommand{\cI}{\mc I}

\newcommand{\cV}{\mc V}

\newcommand{\ms}{\mathscr}

\newcommand{\sS}{\ms S}

\newcommand{\HH}{\mathbb{H}}

\newcommand{\C}{\mathbb{C}}
\newcommand{\N}{\mathbb{N}}
\newcommand{\R}{\mathbb{R}}
\newcommand{\Z}{\mathbb{Z}}

\newcommand{\Sph}{\mathbb{S}}

\newcommand{\ran}{\operatorname{ran}}

\renewcommand{\Re}{\operatorname{Re}}
\renewcommand{\Im}{\operatorname{Im}}

\newcommand{\supp}{\operatorname{supp}}

\newcommand{\diag}{\operatorname{diag}}

\newcommand{\Ups}{\Upsilon}

\newcommand{\eps}{\epsilon}
\newcommand{\ff}{\mathrm{ff}}
\newcommand{\fff}{\mathrm{fff}}

\newcommand{\hra}{\hookrightarrow}
\newcommand{\la}{\langle}

\newcommand{\extcup}{\operatorname{\ol\cup}}

\newcommand{\ol}{\overline}
\newcommand{\pa}{\partial}
\newcommand{\dd}{{\mathrm d}}
\newcommand{\ra}{\rangle}

\newcommand{\Specb}{\operatorname{Spec}_\bop}

\newcommand{\wh}{\widehat}
\newcommand{\wt}{\widetilde}
\newcommand{\xra}{\xrightarrow}
\newcommand{\pfstep}[1]{$\bullet$\ \underline{\textit{#1}}}

\newcommand{\bop}{{\mathrm{b}}}

\newcommand{\scop}{{\mathrm{sc}}}

\newcommand{\scl}{{\mathrm{sc}}}

\newcommand{\lb}{{\mathrm{lb}}}
\newcommand{\rb}{{\mathrm{rb}}}

\newcommand{\bface}{{\mathrm{bf}}}

\newcommand{\cp}{{\mathrm{c}}}

\newcommand{\Diff}{\mathrm{Diff}}

\newcommand{\Psib}{\Psi_\bop}

\newcommand{\Diffsc}{\Diff_\scl}

\newcommand{\Omegab}{{}^{\bop}\Omega}

\newcommand{\Omegazero}{{}^0\Omega}

\newcommand{\KD}{\mathrm{KD}}

\newcommand{\Tzero}{{}^0T}

\newcommand{\half}{{\tfrac{1}{2}}}

\newcommand{\sigmazero}{{}^0\sigma}

\newcommand{\loc}{{\mathrm{loc}}}
\newcommand{\CI}{\cC^\infty}
\newcommand{\CIdot}{\dot\cC^\infty}

\newcommand{\CIc}{\cC^\infty_\cp}

\newcommand{\phg}{{\mathrm{phg}}}

\newcommand{\openbigpmatrix}[1]
  {%
    \def\@bigpmatrixsize{#1}%
    \addtolength{\arraycolsep}{-#1}%
    \begin{pmatrix}%
  }
\newcommand{\closebigpmatrix}
  {%
    \end{pmatrix}%
    \addtolength{\arraycolsep}{\@bigpmatrixsize}%
  }

\newlength{\enummargin}

\newcommand{\usref}[1]{{\upshape\ref{#1}}}

\DeclareGraphicsExtensions{.mps}

\makeatletter
\newcommand*{\fwbw}[1]{\expandafter\@fwbw\csname c@#1\endcsname}
\newcommand*{\@fwbw}[1]{\ifcase #1 \or {\rm fw}\or {\rm bw}\fi}
\AddEnumerateCounter{\fwbw}{\@fwbw}
\makeatother

\begin{document}

\title[Elliptic parametrices in the 0-calculus]{Elliptic parametrices in the 0-calculus of Mazzeo and Melrose}

\date{\today}

\begin{abstract}
  The purpose of this note is to spell out the details of the construction of parametrices for fully elliptic uniformly degenerate pseudodifferential operators on manifolds $X$ with boundary. Following the original work by Mazzeo--Melrose on the 0-calculus, the parametrices are shown to have (polyhomogeneous) conormal Schwartz kernels on the 0-double space, a resolution of $X^2$. The extended 0-double space introduced by Lauter plays a useful role in the construction.
\end{abstract}


\author{Peter Hintz}
\address{Department of Mathematics, ETH Z\"urich, R\"amistrasse 101, 8092 Z\"urich, Switzerland}
\email{peter.hintz@math.ethz.ch}

\maketitle


\section{Introduction}
\label{SI}

In this note, we present complete details for the construction of elliptic parametrices in the 0-calculus. The reader is assumed to have some basic familiarity with blow-up constructions and polyhomogeneous conormal distributions; we refer the reader to \cite{MelrosePushfwd,GrieserBasics,MelroseDiffOnMwc}, \cite[\S2A]{MazzeoEdge}, and further references throughout the paper for the relevant background.

Let $n\geq 1$, and let $X$ be a smooth $n$-dimensional manifold with boundary $\pa X\neq\emptyset$. Following Mazzeo--Melrose \cite{MazzeoMelroseHyp}, the Lie algebra of \emph{uniformly degenerate vector fields} (or \emph{0-vector fields}) is defined as
\[
  \cV_0(X):=\{ V\in\cV(X)\colon V=0\ \text{at}\ \pa X\}.
\]
In local coordinates $x\geq 0$ and $y\in\R^{n-1}$ near a point in $\pa X$, 0-vector fields are linear combinations of $x\pa_x$ and $x\pa_{y^j}$ ($j=1,\ldots,n-1$) with smooth coefficients. For $m\in\N$, the space $\Diff_0^m(X)$ of uniformly degenerate differential operators consists of all $m$-th order operators which are locally finite sums of up to $m$-fold compositions of 0-vector fields.

The spectral family $\Delta_{\HH^n}-\lambda$, $\lambda\in\C$, of the Laplace operator on the Poincar\'e disc model of hyperbolic n-space $\HH^n$ is an example of an elliptic element of $\Diff_0^2(X)$ where $X$ is the closed unit ball (the conformal compactification of the Poincar\'e disc). Indeed, in suitable local coordinates $x\geq 0$, $y\in\R^{n-1}$ corresponding to the upper half plane model of hyperbolic space, one has $\Delta_{\HH^n}-\lambda=(x D_x)^2+i(n-1)x D_x+\sum_{j=1}^{n-1}(x D_{y^j})^2-\lambda$. Writing $\lambda=\zeta(n-1-\zeta)$, Mazzeo--Melrose \cite{MazzeoMelroseHyp} study the meromorphic continuation of the resolvent $(\Delta_g-\zeta(n-1-\zeta))^{-1}$ in $\zeta$ from $\Re\zeta\gg 1$ for asymptotically hyperbolic metrics $g$ on compact manifolds $X$ with boundary; these are suitable variable coefficient generalizations of the exact hyperbolic metric. The core of \cite{MazzeoMelroseHyp} is the development of a calculus of 0-pseudodifferential operators in which rather precise approximate parametrices (or exact inverses if they exist) of the spectral family can be constructed.

The spectral family of asymptotically hyperbolic Laplacians features prominently in recent works on the asymptotic behavior of solutions of the wave equation on spacetimes equipped with suitable asymptotically Minkowskian metrics, such as the Lorentzian scattering metrics of \cite{BaskinVasyWunschRadMink,BaskinVasyWunschRadMink2}. Roughly speaking, when foliating the cone $t>r$ in an asymptotically Minkowski spacetime by the hyperboloidal level sets of $s=(t^2-r^2)^{1/2}$, the wave operator is, approximately, conformally related to the wave operator on the hyperbolic space obtained by restricting a rescaling of the spacetime metric to $s=s_0$ and letting $s_0\nearrow\infty$. On exact $(n+1)$-dimensional Minkowski space, the wave operator is indeed equal to $s^{-2}(-(s D_s)^2+i(n-1)s D_s+\Delta_{\HH^n})$. Upon passing to the Mellin transform in $s$, the properties of the spectral family (in particular the location of its resonances) determine the asymptotic behavior of waves in the forward light cone. In the works \cite{BaskinVasyWunschRadMink,BaskinVasyWunschRadMink2}, a direct analysis of the asymptotically hyperbolic resolvent can be avoided since the spacetimes of interest have a global approximate dilation-invariance in $(t,r)$, including near the light cone $t=r$. On the other hand, on asymptotically flat spacetimes that arise in the context of Einstein's vacuum equations in general relativity, dilation-invariance fails near the light cone, and a direct analysis of a hyperbolic resolvent becomes necessary when studying wave asymptotics in $t>r$. In \cite{HintzVasyMink4}, following \cite{VasyMicroKerrdS,VasyMinkDSHypRelation}, this was accomplished, in the context of a carefully designed (near the light cone) wave equation on symmetric 2-tensors, via an extension across the conformal boundary; see \cite{ZworskiRevisitVasy}, \cite[\S5]{DyatlovZworskiBook}, and \cite[\S3]{ZworskiResonanceReview} for detailed accounts of this procedure, and for some historical context of scattering theory on (asymptotically) hyperbolic manifolds. The present note arose out of the desire to relax the requirements on the wave equation near the light cone, in which case an extension across the boundary can no longer be performed; this is used in ongoing work on the analysis of the Einstein equation on asymptotically flat spacetimes, including Kerr and Schwarzschild spacetimes.

Here, we will revisit the construction of parametrices for arbitrary \emph{fully elliptic 0-pseudo\-differential operators}; these notions will be defined below. A construction of parametrices for edge differential operators---which generalize 0-differential operators---was given by Mazzeo \cite{MazzeoEdge}. As we will discuss below, Albin \cite{AlbinLectureNotes} described the 0-pseudodifferential case in some detail, and Lauter \cite{LauterPsdoConfComp} constructed rough parametrices in this generality as well. The detailed construction of a parametrix with a polyhomogeneous conormal Schwartz kernel in the present note requires a number of technical ingredients which appear here for the first time.

We now return to the general setup. The space $\cV_0(X)$ is equal to the space $\CI(X;\Tzero X)$ of smooth sections of the \emph{0-tangent bundle} $\Tzero X\to X$; in local coordinates $x\geq 0$, $y\in\R^{n-1}$ as above, a local frame for $\Tzero X$ is given by the vector fields $x\pa_x$, $x\pa_{y^j}$ ($j=1,\ldots,n-1$). Denote by $\Omegazero^\alpha X\to X$ the corresponding bundle of $\alpha$-densities, with local frame given by $|\frac{\dd x}{x}\frac{\dd y}{x^{n-1}}|^\alpha=|\frac{\dd x}{x}\frac{\dd y^1\cdots\dd y^{n-1}}{x^{n-1}}|^\alpha$.

With $\diag_{\pa X}=\{(p,p)\colon p\in\pa X\}\subset X$ denoting the boundary diagonal, the \emph{0-double space} is defined as the real blow-up
\[
  X^2_0 := [ X^2; \diag_{\pa X} ].
\]
The boundary hypersurfaces of $X^2_0$ are denoted $\lb$ (the lift of $\pa X\times X)$, $\rb$ (the lift of $X\times\pa X$) and $\ff$ (the front face). The lift of the diagonal $\diag_X\subset X$ to $X^2_0$ is the \emph{0-diagonal}
\[
  \diag_0 := \overline{\diag_X\setminus\diag_{\pa X}} \subset X^2_0;
\]
it is a p-submanifold, i.e.\ given in suitable local coordinates on $X^2_0$ by the vanishing of a subset of the coordinates. See Figure~\ref{FigI0}.

\begin{figure}[!ht]
\centering
\includegraphics{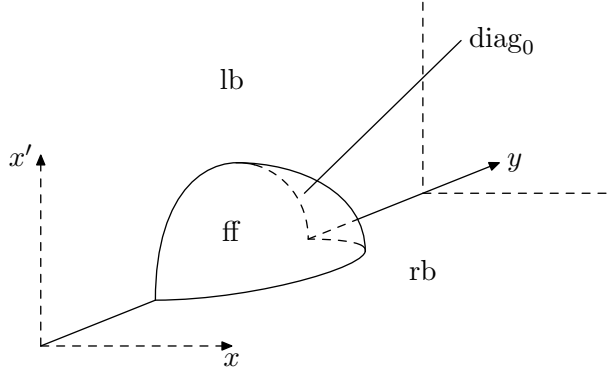}
\caption{The 0-double space $X^2_0$ in a local coordinate chart $x,y,x',y'$ with $y'$ fixed and suppressed, see also the notation in~\S\ref{SsCN}. The boundary diagonal which was blown up here is $(x,y,x',y')=(0,y',0,y')$.}
\label{FigI0}
\end{figure}

Schwartz kernels of 0-vector fields, or more generally of elements of $\Diff_0^m(X)$, lift to distributional right densities on $X^2_0$ which are conormal to $\diag_0$ smoothly (up to a weight factor) down to $\ff_0$; in fact, they are sums of differentiated $\delta$-distributions supported at $\diag_0$. Conversely, any such distribution is the Schwartz kernel of a 0-differential operator, as follows from the fact that the lift of $\cV_0(X)$ to the left factor of $X^2$ and then to $X^2_0$ is transversal to $\diag_0$, see \cite[(4.15)]{MazzeoMelroseHyp}. This fact also implies that one can identify
\begin{equation}
\label{EqIConormal}
  N^*\diag_0\cong\Tzero^* X.
\end{equation}

\begin{definition}[0-pseudodifferential operators]
\label{DefI0}
  Define the kernel density bundle
  \begin{equation}
  \label{EqIKD}
    \KD_0 := \pi_L^*(\Omegazero^{\frac12}X) \otimes \pi_R^*(\Omegazero^{\frac12}X),
  \end{equation}
  where $\pi_L$ and $\pi_R\colon X^2_0\to X$ are the lifts of the left and right projections $(p,q)\mapsto p$ and $(p,q)\mapsto q$. Then the space of 0-pseudodifferential operators (acting on 0-$\half$-densities, i.e.\ uniformly degenerate $\half$-densities, on $X$) is defined on the level of Schwartz kernels as
  \[
    \Psi_0^m(X,\Omegazero^{\frac12}X) := \{ \kappa \in I^m(X^2_0,\diag_0,\KD_0) \colon \kappa \equiv 0\ \text{at}\ \lb\cup\rb \}.
  \]
  Here, `$\equiv 0$' means vanishing in Taylor series, and $I^m$ is the space of conormal distributions of order $m$.
\end{definition}

In view of~\eqref{EqIConormal}, the principal symbol map for 0-ps.d.o.s fits into the short exact sequence\footnote{The principal symbol is valued in the bundle $\Omega_{\rm fiber}(N^*\diag_0)\otimes(\KD_0)|_{\diag_0}$, where $\Omega_{\rm fiber}(N^*\diag_0)$ is the bundle of translation-invariant densities along each fiber of $N^*\diag_0$. In view of~\eqref{EqIConormal}, we have a natural isomorphism $\Omega_{\rm fiber}(N^*\diag_0)\cong\Omega_{\rm fiber}(\Tzero^*X)$. Since $\Tzero X\cong N\diag_0$, we can moreover identify $(\KD_0)|_{\diag_0}\cong\Omegazero^{\frac12}(X)\otimes\Omega_{\rm fiber}^{\frac12}(\Tzero X)$. Altogether then, the principal symbol is valued in $\Omega_{\rm fiber}^{\frac12}(\Tzero^*X)\otimes\Omegazero^{\frac12}(X)$, which has a canonical nonvanishing section given by the symplectic form.}
\[
  0 \to \Psi_0^{m-1}(X,\Omegazero^{\frac12}X) \hra \Psi_0^m(X,\Omegazero^{\frac12}X) \xra{\sigmazero^m} (S^m/S^{m-1})(\Tzero^*X) \to 0.
\]

An operator $P\in\Psi_0^m(X,\Omegazero^{\frac12}X)$ is elliptic (in the symbolic sense) if its principal symbol $\sigmazero^m(P)$ is elliptic. In order for $P$ to be Fredholm on weighted function spaces, a further assumption on its behavior at $\pa X$ is needed; we proceed to introduce the relevant objects in the case that $P\in\Diff_0^m(X,\Omegazero^{\frac12}X)$ is a differential operator. In local coordinates, and fixing a trivialization of $\Omegazero^{\frac12}X$ (e.g.\ using the section $|\frac{\dd x}{x}\frac{\dd y}{x^{n-1}}|^{\frac12}$), we can write $P$ as
\begin{equation}
\label{EqIOp}
  P = \sum_{j+|\beta|\leq m} a_{j\beta}(x,y) (x D_x)^j (x D_y)^\beta.
\end{equation}
(Its principal symbol is $\sum_{j+|\beta|=m}a_{j\beta}(x,y)\xi^j\eta^\beta$, where $(\xi,\eta)\in{}^0 T^*_{(x,y)}X$.) Freezing the coefficients of $P$ at $(0,y_0)\in\pa X$ and exploiting the translation-invariance in $y$ of the resulting operator by passing to the Fourier transform in $y$ gives the \emph{transformed normal operator}
\begin{equation}
\label{EqIOpNorm}
  N(P,y_0,\tilde\eta) = \sum_{j+|\beta|\leq m} a_{j\beta}(0,y_0) (\tilde x D_{\tilde x})^j (\tilde x\tilde\eta)^\beta.
\end{equation}
We use $\tilde x$ here instead of $x$ as we consider $N(P,y_0,\tilde\eta)$ as acting on a model space $[0,\infty)_{\tilde x}$. The scaling invariance $(\tilde x,\tilde\eta)\mapsto(\lambda\tilde x,\tilde\eta/\lambda)$ for $\lambda>0$ of the operator~\eqref{EqIOpNorm} can be further exploited by passing to the \emph{reduced normal operator}
\begin{equation}
\label{EqIOpNormRed}
  \hat N(P,y_0,\hat\eta) = \sum_{j+|\beta|\leq m} a_{j\beta}(0,y_0) (t D_t)^j (t\hat\eta)^\beta,\qquad
  t = \tilde x|\tilde\eta|,\quad
  \hat\eta = \frac{\tilde\eta}{|\tilde\eta|}.
\end{equation}
On the compactified positive half line $[0,\infty]_t$ (i.e.\ using $T=t^{-1}$ as a defining function of $\infty$), this is a weighted b-scattering differential operator, i.e.\ near $t=0$ based on the b-vector field $t\pa_t$ \cite{MelroseAPS} and near $T=0$ on the scattering vector field $T^2\pa_T=-\pa_t$ \cite{MelroseEuclideanSpectralTheory} with an overall weight of $t^m=T^{-m}$. We record this as
\begin{equation}
\label{EqINormRescType}
  \hat N(P,y_0,\hat\eta) \in \Diff_{\bop,\scop}^{m,(0,m)}([0,\infty]),
\end{equation}
with smooth dependence on $(y_0,\hat\eta)\in\R^{n-1}\times\Sph^{n-2}$ (which can be identified with $S^*\pa X$, if desired, upon fixing a collar neighborhood of $\pa X$ and choosing a Riemannian metric on $X$). The operator~\eqref{EqINormRescType} is elliptic, including in the sense of decay at $T=0$.\footnote{In $T<1$, one has
\[
  T^m \hat N(P,y_0,\hat\eta)\equiv \sum_{j+|\beta|=m} a_{j\beta}(0,y_0)\hat\eta^\beta D_t^j\bmod T^{-m+1}\Diffsc^{m-1}([0,1)_T)\qquad (D_t=-T^2 D_T),
\]
which due to the (symbolic) ellipticity of $P$ shows that $\hat N(P,y_0,\hat\eta)$ is indeed fully elliptic as a weighted scattering differential operator.} Any tempered element in its nullspace is thus automatically smooth in $(0,\infty)$ and Schwartz as $t\to\infty$. 

Near $t=0$ on the other hand, the behavior of tempered elements of the kernel of the Fuchsian operator $\hat N(P,y_0,\hat\eta)$ is controlled by its \emph{indicial operator} (or b-normal operator)
\[
  I(P,y_0) := \sum_{j\leq m} a_{j 0}(0,y_0)(t D_t)^j
\]
and, via the Mellin transform in $t$, by the \emph{indicial family}
\[
  I(P,y_0,\sigma) := \sum_{j\leq m} a_{j 0}(0,y_0)\sigma^j,\qquad \sigma\in\C.
\]
(Note that the leading order coefficient $a_{m 0}(0,y_0)$ is nonzero.) As reflected in the notation, $I(P,y_0)$ and $I(P,y_0,\sigma)$ do not depend on $\hat\eta$. At each $y_0\in\pa X$, we define the boundary spectrum by
\begin{equation}
\label{EqISpecb}
  \Specb(P,y_0) := \{ (z,k) \in \C\times\N_0 \colon I(P,y_0,\sigma)^{-1}\ \emph{has a pole at $\sigma=-i z$ of order $\geq k+1$} \}.
\end{equation}
Elements in the nullspace of $\hat N(P,y_0,\hat\eta)$ have asymptotic expansions at $t=0$ involving terms\footnote{The introduction of the factor $i$ in~\eqref{EqISpecb} eliminates a factor of $i$ here.} $t^z(\log t)^k$ for $(z,k)\in\Specb(P,y_0)$ (as well as further terms with $z$ increased by integers, and possibly larger $k$ arising from integer coincidences, i.e.\ from the existence of $(z_1,0)$ and $(z_2,0)\in\Specb(P,y_0)$ with $z_2-z_1\in\Z\setminus\{0\}$).

In~\S\ref{SC}, we will discuss the definitions of the (reduced) normal operator and indicial operator/family and boundary spectrum for \emph{pseudo}differential operators $P\in\Psi_0^m(X,\Omegazero^{\frac12}X)$. In particular, in this generality $\hat N(P,y_0,\hat\eta)\in\Psi_{\bop,\scop}^{m,(0,m)}([0,\infty],\Omegab^{\frac12}[0,\infty])$ is an elliptic weighted b-scattering-pseudodifferential operator.

\begin{definition}[Invertibility of the reduced normal operator]
\label{DefIInv}
  Let $\alpha\in\R$, and let $P\in\Psi_0^m(X,\Omegazero^{\frac12}X)$ be an operator with elliptic principal symbol. We say that $\hat N(P,y_0,\hat\eta)$ is \emph{invertible at the weight $\alpha$} if the following three conditions hold:
  \begin{enumerate}
  \item $\alpha\neq\Re z$ for all $(z,k)\in\Specb(P,y_0)$.
  \item If $u\in t^\alpha L^2([0,\infty],|\frac{\dd t}{t}|)$ solves $\hat N(P,y_0,\hat\eta)u=0$, then $u=0$.
  \item If $v\in t^{-\alpha}L^2([0,\infty],|\frac{\dd t}{t}|)$ solves $\hat N(P,y_0,\hat\eta)^*v=0$, then $v=0$. Here, the adjoint is defined with respect to the volume density $|\frac{\dd t}{t}|$, or more generally with respect to any polynomially weighted (at infinity) b-volume density $a(t)(1+t)^{-w}|\frac{\dd t}{t}|$ where $0<a\in\CI([0,\infty])$, $w\in\R$.
  \end{enumerate}
  If $\hat N(P,y_0,\hat\eta)$ is invertible at the weight $\alpha$ for all $y_0,\hat\eta$, then we say that $P$ is \emph{fully elliptic at the weight $\alpha$}.
\end{definition}

\begin{rmk}[Invertibility as an operator between Sobolev spaces]
\label{RmkISob}
  The invertibility at the weight $\alpha$ can be phrased slightly more naturally as the invertibility of $\hat N(P,y_0,\hat\eta)$ acting between suitable weighted b-scattering-Sobolev spaces,
  \[
    \hat N(P,y_0,\hat\eta) \colon H_{\bop,\scop}^{s,\alpha,r}([0,\infty])=\Bigl(\frac{t}{t+1}\Bigr)^\alpha(t+1)^{-r}H_{\bop,\scop}^s\Bigl([0,\infty];\Bigl|\frac{\dd t}{t}\Bigr|\Bigr)\to H_{\bop,\scop}^{s-m,\alpha,r-m}([0,\infty]),
  \]
  for any $s,r\in\R$. (By ellipticity, the invertibility is independent of the choice of $s,r$.) See (the proof of) Proposition~\ref{PropPNInv}.
\end{rmk}

In order to state the main result, we need to introduce the \emph{large 0-calculus}. Recall first that an \emph{index set} is a subset $\cE\subset\C\times\N_0$ with the property that $(z,k)\in\cE$ implies $(z+1,k)\in\cE$ and (when $k\geq 1$) $(z,k-1)\in\cE$, and so that moreover for any $C$ the set $\{(z,k)\in\cE\colon\Re z<C\}$ is finite.\footnote{Other conventions, with factors of $\pm i$ in the first component, are also frequently used in the literature. Our present convention is chosen to match~\eqref{EqISpecb}.} Recall moreover that spaces of polyhomogeneous conormal distributions are local, and it suffices to define the space $\cA_\phg^{(\cE_1,\ldots,\cE_k)}([0,\infty)_x^k\times\R_y^{n-k})$ for index sets $\cE_1,\ldots,\cE_k\in\C\times\N_0$. This space can be defined by induction over $k$ \cite[\S4.13]{MelroseDiffOnMwc}, and consists of distributions which at $x_1=0$ are asymptotic sums of terms $x_1^z(\log x_1)^k a_{z,k}(x',y)$, $(z,k)\in\cE_1$, $x'=(x_2,\ldots,x_n)$, where the $a_{z,k}$ themselves lie in $\cA_\phg^{(\cE_2,\ldots,\cE_k)}([0,\infty)_{x'}^{k-1}\times\R_y^{n-k})$; similarly at all other boundary hypersurfaces $x_j=0$.

\begin{definition}[Full 0-calculus]
\label{DefIFull}
  For a collection $\cE=(\cE_\lb,\cE_\ff,\cE_\rb)$ of index sets $\cE_\lb$, $\cE_\ff$, $\cE_\rb\subset\C\times\N_0$, we define the space of \emph{residual operators}
  \[
    \Psi_0^{-\infty,\cE}(X) := \cA_\phg^\cE(X^2_0,\KD_0),
  \]
  with the index set $\cE_H$ associated with the boundary hypersurface $H\subset X^2_0$. More generally for $m\in\R\cup\{-\infty\}$, put
  \[
    \Psi_0^{m,\cE}(X) := \cA_\phg^\cE I^m(X_0^2,\diag_0,\KD_0),
  \]
  i.e.\ where the conormal distribution to $\diag_0$ has coefficients which are polyhomogeneous down to $\ff$ with index set $\cE_\ff$. For $\cE'=(\cE_0,\cE_1)$, we moreover define the space of \emph{fully residual operators}
  \[
    \Psi^{-\infty,\cE'}(X) := \cA_\phg^{\cE'}(X^2,\KD_0),
  \]
  with the index set $\cE_0$, resp.\ $\cE_1$ associated with the boundary hypersurface $\pa X\times X$, resp.\ $X\times\pa X$ of $X\times X$.
\end{definition}

Recall also that the sum and extended union of two index sets $\cE,\cF$ are defined as
\begin{align*}
  \cE+\cF &:= \{(z_1+z_2,k_1+k_2) \colon (z_1,k_1)\in\cE,\ (z_2,k_2)\in\cF \}, \\
  \cE\extcup\cF &:= \cE\cup\cF\cup\{(z,k_1+k_2+1)\colon(z,k_1)\in\cE,\ (z,k_2)\in\cF\}.
\end{align*}
We moreover set $\cE+j=\cE+\{(j',0)\colon j'\in\N_0,\ j'\geq j\}$. With this setup, our main result is:

\begin{thm}[Parametrices of fully elliptic 0-ps.d.o.s]
\label{ThmI}
  Let $P\in\Psi_0^m(X,\Omegazero^{\frac12}X)$ be fully elliptic at the weight $\alpha$, and assume that $\Specb(P,y_0)$ is independent of $y_0\in\pa X$; put $\Specb(P):=\Specb(P,y_0)$. Denote by $\cE_\pm\subset\C\times\N_0$ the smallest index sets\footnote{The existence of these index sets is proved in Corollary~\ref{CorCNInd}.} with
  \begin{equation}
  \label{EqIEpm}
  \begin{split}
    \cE_+ &\supset \{ (z,k) \colon (z,k)\in\Specb(P),\ \Re z>\alpha \}, \\
    \cE_- &\supset \{ (-z,k) \colon (z,k)\in\Specb(P),\ \Re z<\alpha \}.
  \end{split}
  \end{equation}
  (Thus, $\Re\cE_+>\alpha$, i.e.\ $\Re z>\alpha$ for all $(z,k)\in\cE_+$, and $\Re\cE_->-\alpha$.) Set $\wh{\cE_\pm}(0):=\cE_\pm$ and $\wh{\cE_\pm}(j):=\cE_\pm\extcup\,(\wh{\cE_\pm}(j-1)+1)$ for $j=1,2,3,\ldots$, and let $\wh{\cE_\pm}:=\bigcup_{j=0}^\infty \wh{\cE_\pm}(j)$. Define similarly $\wh{\cE_\pm^\flat}:=\wh{\cE_\pm}\extcup\,(\wh{\cE_\pm}+1)\extcup\cdots$, and put $\wh{\cE_\pm^\sharp}:=\wh{\cE_\pm^\flat}\extcup\,(\wh{\cE_\pm}+1)$, $\wh{\cE_\ff^\pm}:=\N_0\extcup\,(\wh{\cE_\mp}+\wh{\cE_\pm^\flat}+(n-1))$. Then there exist parametrices
  \begin{align*}
    Q&\in\Psi_0^{-m}(X,\Omegazero^{\frac12}X)+\Psi_0^{-\infty,(\wh{\cE_+},\wh{\cE_\ff^-},\wh{\cE_-^\sharp}+(n-1))}(X,\Omegazero^{\frac12}X), \\
    Q'&\in\Psi_0^{-m}(X,\Omegazero^{\frac12}X)+\Psi_0^{-\infty,(\wh{\cE_+^\sharp},\wh{\cE_\ff^+},\wh{\cE_-}+(n-1))}(X,\Omegazero^{\frac12}X),
  \end{align*}
  so that
  \begin{alignat*}{2}
    P Q &= I-R, &\qquad R&\in\Psi^{-\infty,(\emptyset,\wh{\cE_-^\flat}+(n-1))}(X,\Omegazero^{\frac12}X), \\
    Q' P &= I-R', &\qquad R'&\in\Psi^{-\infty,(\wh{\cE_+^\flat},\emptyset)}(X,\Omegazero^{\frac12}X).
  \end{alignat*}
\end{thm}

For compact $X$, this implies that $P$ is Fredholm as a map between weighted 0-Sobolev spaces $\rho^\alpha H_0^s(X,\Omegazero^{\frac12}X)\to\rho^\alpha H_0^{s-m}(X,\Omegazero^{\frac12}X)$ (with $\rho\in\CI(X)$ a boundary defining function). Its generalized inverse moreover has the same structure as the parametrices in Theorem~\ref{ThmI} with slightly different index sets at $\lb$ and $\rb$; see Remark~\ref{RmkPrb} and Corollary~\ref{CorPGen}. The index sets of $Q$, resp.\ $Q'$ at $\ff$ and $\rb$, resp.\ $\lb$ are likely much bigger than needed in order for the error terms $R,R'$ to be fully residual with trivial index set at one of the two boundaries of $X^2$; we do not aim to optimize them here.

When $P$ is an element of the spectral family of the Laplacian on an asymptotically hyperbolic manifold, with spectral parameter lying in the resolvent set, Theorem~\ref{ThmI} gives the same information on the resolvent kernel as \cite{MazzeoMelroseHyp} (except with less precise index sets). Note however that the meromorphic continuation of the resolvent in the spectral parameter (away from an exceptional set \cite{GuillarmouMeromorphic}) proved in \cite{MazzeoMelroseHyp} goes far beyond Theorem~\ref{ThmI}, since the parametrix construction for individual 0-operators provided by Theorem~\ref{ThmI} cannot distinguish between the cases that the spectral parameter lies in the resolvent set or in the continuation region.

Albin's lecture notes \cite{AlbinLectureNotes} give a rather detailed account of the proof of Theorem~\ref{ThmI}; we largely follow the arguments of \cite[\S5.5]{AlbinLectureNotes} and supplement them with the necessary auxiliary technical results. The first step of the parametrix construction is purely symbolic, i.e.\ uses the invertibility of the principal symbol of $P$. The next step is the inversion of the reduced normal operator; showing that the inverse is itself in the range of the reduced normal operator is a subtle task due to the absence of a simple description of the range of the reduced normal operator map in general. See Proposition~\ref{PropPNInv}. The main work in this step is the construction of left and right parametrices of the reduced normal operator; the nontrivial part is the full parametrix construction in the b-calculus \cite{MelroseAPS} while remaining in the range of the reduced normal operator map. Given the normal operator inverse, the construction of a full right 0-parametrix follows a standard procedure (solving away the remaining error to infinite order at $\lb$ using indicial operator arguments, and then using an asymptotic Neumann series to solve away the remaining error at $\ff$ to infinite order).

When the boundary spectrum $\Specb(P,y_0)$ does depend on $y_0$, Schwartz kernels of similarly precise parametrices are no longer polyhomogeneous, but one can still prove their conormality (with rather sharp bounds) on $X^2_0$ under a mild gap condition on the boundary spectra; see Theorem~\ref{ThmB}. In this generality, Lauter \cite{LauterPsdoConfComp} constructs a parametrix on the extended 0-double space which is sufficiently precise to deduce Fredholm properties of $P$. (In the notation of Theorem~\ref{ThmI}, Lauter's error terms $R_L$ and $R_R$ are residual operators which vanish to some small positive order at the hypersurfaces $\ff'\cup\ff_\bop$ of the extended 0-double space, cf.\ Figure~\ref{FigCExt}.)

\begin{rmk}[Bundles]
  We shall only consider operators acting between $\half$-densities. The analysis of operators acting between sections on bundles requires purely notational changes.
\end{rmk}

\begin{rmk}[Generalizations]
  A detailed parametrix construction for fully elliptic pseudodifferential operators in Mazzeo's edge calculus \cite{MazzeoEdge} can be given in much the same way as in the present note. In particular, the reduced normal operator is still a b-scattering operator, just not on $[0,\infty]$ but rather on the product of $[0,\infty]$ with the typical fiber of the boundary fibration.
\end{rmk}

\subsection*{Acknowledgments}

I gratefully acknowledge support from the U.S.\ National Science Foundation under Grant No.\ DMS-1955416 and from a Sloan Research Fellowship.

\section{The (extended) 0-calculus}
\label{SC}

The (large) 0-calculus was already introduced in Definitions~\ref{DefI0} and \ref{DefIFull}. Following \cite{LauterPsdoConfComp}, defining \emph{extended 0-calculus} requires a further resolution of the 0-double space:

\begin{definition}[Extended 0-double space]
\label{DefCExt}
  The extended 0-double space\footnote{This is denoted $X^2_{0,e}$ in \cite{LauterPsdoConfComp}; we use a slightly different notation here to avoid confusion with the notation for the edge double space in \cite{MazzeoEdge}.} $X^2_{0'}$ is defined as the iterated blow-up
  \[
    X^2_{0'} := \bigl[ X^2; (\pa X)^2; \diag_{\pa X} \bigr].
  \]
  Its boundary hypersurfaces are denoted $\lb'$ (lift of $\pa X\times X$), $\ff'$ (lift of $\diag_{\pa X}$), $\rb'$ (lift of $X\times\pa X$), and $\ff_\bop$ (lift of $(\pa X)^2$).
\end{definition}

See Figure~\ref{FigCExt}.

\begin{figure}[!ht]
\centering
\includegraphics{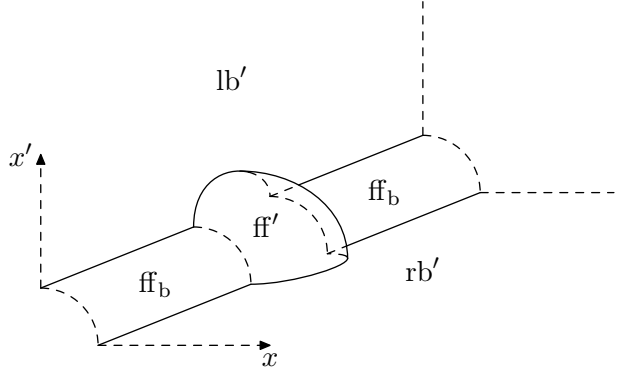}
\caption{The extended 0-double space $X^2_{0'}$.}
\label{FigCExt}
\end{figure}

Since $\diag_{\pa X}\subset(\pa X)^2$, the order of the two blow-ups is immaterial in the sense that the identity map on $(X^\circ)^2$ extends to a diffeomorphism
\[
  [X^2;(\pa X)^2;\diag_{\pa X}]\xra{\cong} [X^2;\diag_{\pa X};(\pa X)^2].
\]
Thus, $X^2_{0'}$ is the blow-up of $X^2_0$ along the lift of $(\pa X)^2$; moreover, $\ff'$ is the blow-up of $\ff$ at the corner $\ff\cap(\lb\cup\rb)$.

Since $X^2_0$ and $X^2_{0'}$ are naturally diffeomorphic near $\diag_0$, we can equivalently define
\begin{equation}
\label{EqCExtPsi}
  \Psi_0^m(X,\Omegazero^{\frac12}X) := \{ \kappa \in I^m(X^2_{0'},\diag_0,\KD_0) \colon \kappa\equiv 0\ \text{at}\ \lb'\cup\ff_{\bop}\cup\rb' \}.
\end{equation}
Here, we abuse notation and write $\KD_0$ also for its lift to $X^2_{0'}$. Spaces of residual operators lift according to
\begin{equation}
\label{EqCRel}
  \Psi_0^{-\infty,(\cE_\lb,\cE_\ff,\cE_\rb)}(X,\Omegazero^{\frac12}X) \subset \Psi_{0'}^{-\infty,(\cE_\lb,\cE_\ff,\cE_\lb+\cE_\rb,\cE_\rb)}(X,\Omegazero^{\frac12}X),
\end{equation}
where the right hand side consist of polyhomogeneous Schwartz kernels with the stated index sets at $\lb'$, $\ff'$, $\ff_\bop$, and $\rb'$ (in this order). Conversely, we note the inclusion
\[
  \Psi_{0'}^{-\infty,(\cE_\lb,\cE_\ff,\emptyset,\cE_\rb)}(X,\Omegazero^{\frac12}X) \subset \Psi_0^{-\infty,(\cE_\lb,\cE_\ff,\cE_\rb)}(X,\Omegazero^{\frac12}X).
\]

We shall also consider less precise classes of residual operators whose Schwartz kernels are merely conormal at some of the boundary hypersurfaces. For a general manifold with corners $M$ and embedded boundary hypersurfaces $H_1$, $\ldots$, $H_N$ with boundary defining functions $\rho_{H_1}$, $\ldots$, $\rho_{H_N}$, we denote by
\[
  \cA^{(\alpha_1,\ldots,\alpha_N)}(M)\qquad (\alpha_1,\ldots,\alpha_N\in\R)
\]
the space of elements of $(\prod_{j=1}^N\rho_{H_j}^{\alpha_j})L_\loc^\infty(M)$ which remain in this space under application of any number of b-vector fields (smooth vector fields on $M$ tangent to all $H_j$). More generally, let $\cH=\{H_i\colon i\in I\}\subset\{H_1,\ldots,H_N\}$ denote a subset of boundary hypersurfaces, where $I\subset\{1,\ldots,N\}$. Consider weights $\beta=(\beta_i\colon i\in I)$ where $\beta_i\in\R$; then
\[
  \cA_\cH^\beta(M)
\]
consists of those elements of $(\prod_{i\in I}\rho_{H_i}^{\beta_i})L_\loc^\infty(M)$ which remain in this space upon application of any number of smooth vector fields on $M$ which are tangent to all elements of $\cH$.

For weights $\alpha_\lb,\alpha_\ff,\alpha_\rb\in\R\cup\{\infty\}$, we now let
\[
  \Psi_0^{m,(\alpha_\lb,\alpha_\rb)}(X,\Omegazero^{\frac12}X) := \Psi_0^m(X,\Omegazero^{\frac12}X) + \Psi_0^{-\infty,(\alpha_\lb,\alpha_\rb)}(X,\Omegazero^{\frac12}X),
\]
where
\begin{align*}
  \Psi_0^{-\infty,(\alpha_\lb,\alpha_\rb)}(X,\Omegazero^{\frac12}X) &:= \cA_{\{\lb,\rb\}}^{(\alpha_\lb,\alpha_\rb)}(X^2_0,\KD_0), \\
  \Psi_0^{-\infty,(\alpha_\lb,\alpha_\ff,\alpha_\rb)}(X,\Omegazero^{\frac12}X) &:= \cA^{(\alpha_\lb,\alpha_\ff,\alpha_\rb)}(X^2_0,\KD_0).
\end{align*}
The \emph{0-calculus with bounds} is then the space of all operators in
\begin{align*}
  \Psi_0^{m,(\alpha_\lb,\alpha_\ff,\alpha_\rb)}(X,\Omegazero^{\frac12}X) &:= \Psi_0^m(X,\Omegazero^{\frac12}X) + \Psi_0^{-\infty,(\alpha_\lb,\alpha_\rb)}(X,\Omegazero^{\frac12}X)\\
    &\hspace{10em} + \Psi_0^{-\infty,(\alpha_\lb,\alpha_\ff,\alpha_\rb)}(X,\Omegazero^{\frac12}X)
\end{align*}
for $m,\alpha_\lb,\alpha_\ff,\alpha_\rb\in\R$. Very residual operators have Schwartz kernels in
\[
  \Psi^{-\infty,(\alpha_0,\alpha_1)}(X,\Omegazero^{\frac12}X) := \cA^{(\alpha_0,\alpha_1)}(X^2,\KD_0).
\]

\subsection{Composition}
\label{SsCC}

For calculations with polyhomogeneous conormal distributions, it is convenient to work with b-densities. We have a natural isomorphism of weighted density bundles $\Omegazero X\cong\rho^{-(n-1)}\,\Omegab X$ where $\rho\in\CI(X)$ is a boundary defining function. If we denote by $\rho_H$ a smooth defining function of a hypersurface $H$, then since $\Omegab X^2$ lifts to $X^2_0$ as $\rho_\ff^{n-1}\,\Omegab X^2_0$, we have
\begin{equation}
\label{EqCDensities}
  \KD_0 \cong \rho_\lb^{-\frac{n-1}{2}}\rho_\rb^{-\frac{n-1}{2}}\rho_\ff^{-\frac{n-1}{2}}\,\Omegab^{\frac12}X^2_0 \cong \rho_{\lb'}^{-\frac{n-1}{2}}\rho_{\rb'}^{-\frac{n-1}{2}}\rho_{\ff'}^{-\frac{n-1}{2}}\rho_{\ff_\bop}^{-\frac{n-1}{2}}\,\Omegab^{\frac12} X^2_{0'}.
\end{equation}

In the 0-calculus, we need the full range of composition results.

\begin{prop}[Composition of 0-ps.d.o.s]
\label{PropCC0}
  Let $\cE=(\cE_\lb,\cE_\ff,\cE_\rb)$ and $\cF=(\cF_\lb,\cF_\ff,\cF_\rb)$ be two collections of index sets, and put $\cG=(\cG_\lb,\cG_\ff,\cG_\rb)$ where
  \begin{gather*}
    \cG_\lb = \cE_\lb \extcup\,(\cE_\ff+\cF_\lb), \qquad
    \cG_\rb = (\cE_\rb+\cF_\ff) \extcup \cF_\rb, \\
    \cG_\ff = (\cE_\ff+\cF_\ff) \extcup\,(\cE_\lb+\cF_\rb).
  \end{gather*}
  Suppose that $\Re(\cE_\rb+\cF_\lb)>n-1$. Then for $m,m'\in\R\cup\{-\infty\}$, we have
  \[
    \Psi_0^{m,\cE}(X,\Omegazero^{\frac12}X) \circ \Psi_0^{m',\cF}(X,\Omegazero^{\frac12}X) \subset \Psi_0^{m+m',\cG}(X,\Omegazero^{\frac12}X).
  \]
\end{prop}

\begin{rmk}[Properly supported Schwartz kernels]
  If $X$ is noncompact, we shall consider compositions of two 0-ps.d.o.s only under the assumption (which we will not spell out henceforth) that the Schwartz kernel of at least one of them is properly supported in $X^2_0$. Note that any $P\in\Psi_0^m(X,\Omegazero^{\frac12}X)$ is equal to the sum of a properly supported operator and an element of $\cA_\phg^\emptyset(X^2,\KD_0)$. Thus, for the purposes of parametrix constructions, we can restrict to the subclass of operators with properly supported Schwartz kernels.
\end{rmk}

\begin{proof}[Proof of Proposition~\usref{PropCC0}]
  We sketch the proof using the machinery of triple spaces and b-fibrations; we shall only consider the case $m=m'=-\infty$. The 0-triple space is
  \begin{equation}
  \label{EqCC0Triple}
    X_0^3 := [ X^3; B'; B_F, B_S, B_C ],
  \end{equation}
  where $B'=\{(p,p,p)\colon p\in\pa X\}$ is the triple diagonal of the boundary, and $B_F$, $B_S$, $B_C$ are the lifts to $[X^3;B']$ of the preimages of $\diag_{\pa X}$ under the projection maps $X^3\to X^2$ given by $\pi_F\colon(p,q,r)\mapsto(p,q)$, $\pi_S\colon(p,q,r)\mapsto(q,r)$, and $\pi_C\colon(p,q,r)\mapsto(p,r)$, respectively. We denote the lift of $B'$ by $\fff$, and the lifts of $B_F,B_S,B_C$ by $\ff_F,\ff_S,\ff_C$; the lifts of $X\times X\times\pa X$, $\pa X\times X\times X$, and $X\times\pa X\times X$ are denoted $\bface_F$, $\bface_S$, and $\bface_C$, respectively. The projection maps $\pi_F,\pi_S,\pi_C$ lift to b-fibrations $X^3_0\to X^2_0$, see \cite[\S4.6]{AlbinLectureNotes}.
  
  Consider now residual 0-ps.d.o.s $A,B$ with Schwartz kernels
  \begin{equation}
  \label{EqCC0SK}
    K_A\in\cA_\phg^\cE(X^2_0,\KD_0) = \cA_\phg^{\cE'}(X^2_0,\Omegab^{\frac12} X_0^2), \qquad
    K_B\in\cA_\phg^{\cF'}(X^2_0,\Omegab^{\frac12}X_0^2),
  \end{equation}
  where, in view of~\eqref{EqCDensities}, $\cE'=(\cE'_\lb,\cE'_\ff,\cE'_\rb):=\cE-(\frac{n-1}{2},\frac{n-1}{2},\frac{n-1}{2})$, similarly for $\cF'$. The Schwartz kernel of $A\circ B$ (if the composition is defined) is the push-forward
  \[
    K_{A\circ B} = \nu^{-1} (\pi_C)_*\bigl( \pi_F^*K_A \cdot \pi_S^*K_B\cdot\pi_C^*\nu \bigr),
  \]
  where $0<\nu\in\CI(X^2_0,\Omegab^{\frac12}X^2_0)$ is arbitrary. Upon ordering the hypersurfaces of $X^3_0$ by $\fff$, $\ff_F,\ff_S,\ff_C$, $\bface_F,\bface_S,\bface_C$, we have
  \begin{equation}
  \label{EqCCProd}
  \begin{split}
    &\pi_F^*K_A \cdot \pi_S^*K_B\cdot\pi_C^*\nu \\
    &\quad \in \cA_\phg^{(\cE'_\ff+\cF'_\ff+\frac{n-1}{2},\ \cE'_\ff+\cF'_\lb+\frac{n-1}{2},\cE'_\rb+\cF'_\ff+\frac{n-1}{2},\cE'_\lb+\cF'_\rb+\frac{n-1}{2},\ \cF'_\rb,\cE'_\lb,\cE'_\rb+\cF'_\lb)}(X^3_0,\Omegab X^3_0).
  \end{split}
  \end{equation}
  The shifts by $\frac{n-1}{2}$ arise due to the relationship between b-$\half$-density bundles upon blowing up a submanifold which has codimension $n-1$ in the smallest boundary face containing it. (This is analogous to the discussion leading up to~\eqref{EqCDensities}.)

  The map $\pi_C$ maps $\fff\cup\ff_C$ to $\ff$, $\bface_S\cup\ff_F$ to $\lb$, and $\bface_F\cup\ff_S$ to $\rb$; push-forward of~\eqref{EqCCProd} along $\pi_C$ is well-defined provided $\Re(\cE'_\rb+\cF'_\lb)>0$. Under this condition,
  \[
    K_{A\circ B} \in \cA_\phg^{(\cE'_\lb\extcup\,(\cE'_\ff+\cF'_\lb+\frac{n-1}{2}),(\cE'_\ff+\cF'_\ff+\frac{n-1}{2})\extcup\,(\cE'_\lb+\cF'_\rb+\frac{n-1}{2}),\cF'_\rb\extcup\,(\cE'_\rb+\cF'_\ff+\frac{n-1}{2}))}(X^2_0,\Omegab^{\frac12}X^2_0).
  \]
  Reverting back to the density $\KD_0$, we obtain $K_{A\circ B}$ as a section of $\KD_0$, with all index sets increased relative to those of $K_{A\circ B}$ by $\frac{n-1}{2}$. This gives the claimed result.
\end{proof}

\begin{lemma}[Action on polyhomogeneous distributions]
\label{LemmaCPhg}
  Let $\cE=(\cE_\lb,\cE_\ff,\cE_\rb)$ be a collection of index sets and $A\in\Psi_0^{m,\cE}(X,\Omegazero^{\frac12}X)$. Let $\cF\subset\C\times\N_0$ denote an index set. Suppose $\Re(\cE_\rb+\cF)>n-1$, and let $\cG=(\cF+\cE_\ff)\extcup\cE_\lb$. Then
  \begin{equation}
  \label{EqCPhg}
    A \colon \cA_\phg^\cF(X,\Omegazero^{\frac12}X) \to \cA_\phg^\cG(X,\Omegazero^{\frac12}X).
  \end{equation}
\end{lemma}
\begin{proof}
  For simplicity of notation, we only consider the case $m=-\infty$. Let $\pi_L,\pi_R\colon X^2_0\to X$ denote the stretched left and right projections. Fix a b-density $0<\nu\in\CI(X,\Omegab^{\frac12}X)$. Write $K_A=\cA_\phg^{\cE'}(X_0^2,\Omegab^{\frac12}X^2_0)$, $\cE'=\cE-(\frac{n-1}{2},\frac{n-1}{2},\frac{n-1}{2})$ for the Schwartz kernel of $A$ as in~\eqref{EqCC0SK}. For $f\in\cA_\phg^\cF(X,\Omegazero^{\frac12}X)=\cA_\phg^{\cF'}(X,\Omegab^{\frac12}X)$, $\cF'=\cF-\frac{n-1}{2}$, we compute
  \[
    A f = \nu^{-1} (\pi_L)_* \bigl( K_A \cdot \pi_R^*f \cdot \pi_L^*\nu \bigr).
  \]
  Note then that the index set at $\rb$ of
  \[
    K_A\cdot\pi_R^*f\cdot\pi_L^*\nu \in \cA_\phg^{(\cE'_\lb,\cE'_\ff+\cF'+\frac{n-1}{2},\cE'_\rb+\cF')}(X^2_0,\Omegab X^2_0)
  \]
  has real part larger than $0$, and hence pushforward along $(\pi_L)_*$ produces an element of $\cA_\phg^{\cE'_\lb\extcup\,(\cE'_\ff+\cF'+\frac{n-1}{2})}(X,\Omegab X)$. Division by $\nu$ and regarding the result as a section of $\Omegazero^{\frac12}X$ gives~\eqref{EqCPhg}.
\end{proof}

\begin{rmk}[Action on polyhomogeneous distributions does not determine $\cE_\ff$]
  The index set $\cG$ can be significantly smaller than stated in Lemma~\ref{LemmaCPhg}. For example, the operator $(x D_{y_1})^N\in\Psi_0^N(X)$---whose Schwartz kernel does \emph{not} vanish at $\ff$---maps $\cA_\phg^\cF(X)\to\cA_\phg^{\cF+N}(X)$. Using the notation~\eqref{Eq0Op}, this is related to the fact that the first $N$ moments of $K_P^0$ in the $\Ups$-variables vanish, or equivalently that its Fourier transform in $\Ups$ vanishes to order $N$ at $\eta=0$. More generally then, one can construct an operator $P\in\Psi_0^{-\infty}(X,\Omegazero^{\frac12}X)$ with nontrivial normal operator but which maps $\cA_\phg^\cF(X)\to\cA_\phg^\emptyset(X)=:\CIdot(X)$ for all $\cF$. 
\end{rmk}

In the extended 0-calculus, we only need the following result:

\begin{prop}[Composition of extended 0-ps.d.o.s]
\label{PropCC0x}
  Let $\cE=(\cE_{\lb'},\cE_{\ff'},\cE_{\ff_\bop},\cE_{\rb'})$ denote a collection of index sets, and let $m\in\R$. Then
  \[
    \Psi_0^m(X,\Omegazero^{\frac12}X) \circ \Psi_{0'}^{-\infty,\cE}(X,\Omegazero^{\frac12}X) \subset \Psi_{0'}^{-\infty,\cE}(X,\Omegazero^{\frac12}).
  \]
\end{prop}
\begin{proof}
  Lifting the Schwartz kernel of an element of $\Psi_0^m(X,\Omegazero^{\frac12}X)$ to $X^2_{0'}$ (see~\eqref{EqCExtPsi}), the proof follows via pullback and pushforward results completely analogously to the proof of Proposition~\ref{PropCC0}. In fact, general compositions in the extended 0-calculus can be analyzed by means of an appropriate extended 0-triple space, defined as
  \[
    X^3_{0'} := \bigl[ X^3; (\pa X)^3; X\times\pa X\times\pa X, \pa X\times X\times\pa X, \pa X\times\pa X\times X; B'; B_F, B_S, B_C \bigr]
  \]
  where $B'$, $B_F$, $B_S$, $B_C$ are as after~\eqref{EqCC0Triple}; the three projections $X^3\to X^2$ lift to b-fibrations $X^3_{0'}\to X^2_{0'}$. We leave the details to the reader.
\end{proof}

In the calculus with bounds, we shall use the following composition result; the (omitted) proof follows again from pullback and pushforward results \cite{MelrosePushfwd}.

\begin{prop}[Composition in the 0-calculus with bounds]
\label{PropCCBounds}
  Given $\alpha=(\alpha_\lb,\alpha_\ff,\alpha_\rb)$ and $\beta=(\beta_\lb,\beta_\ff,\beta_\rb)\in\R^3$, fix $\gamma=(\gamma_\lb,\beta_\ff,\beta_\rb)\in\R^3$ so that
  \begin{equation}
  \label{EqCCBounds}
  \begin{split}
    &\gamma_\lb \leq \min(\alpha_\lb,\alpha_\ff+\beta_\lb),\quad
    \gamma_\rb \leq \min(\alpha_\rb+\beta_\ff,\beta_\rb), \quad
    \gamma_\ff \leq \min(\alpha_\ff+\beta_\ff,\alpha_\lb+\beta_\rb), \\
    &\quad\text{with strict inequality in each individual case when both arguments of $\min$ are equal.}
  \end{split}
  \end{equation}
  Suppose that $\alpha_\rb+\beta_\lb>n-1$. Then for $m,m'\in\R\cup\{-\infty\}$, we have
  \[
    \Psi_0^{m,\alpha}(X,\Omegazero^{\frac12}X) \circ \Psi_0^{m',\beta}(X,\Omegazero^{\frac12}X) \subset \Psi_0^{m+m',\gamma}(X,\Omegazero^{\frac12}X).
  \]
  When $\alpha_\ff=0$, $\beta_\ff=0$, let $\gamma_\lb,\gamma_\rb$ be as in~\eqref{EqCCBounds}, and let $\gamma_\ff=\alpha_\lb+\beta_\rb$ unless $\alpha_\lb+\beta_\rb=k\in\N_0$ in which case fix $\gamma_\ff<k$. Then
  \begin{align*}
    &\Psi_0^{m,(\alpha_\lb,\alpha_\rb)}(X,\Omegazero^{\frac12}X) \circ \Psi_0^{m',(\beta_\lb,\beta_\rb)}(X,\Omegazero^{\frac12}X) \\
    &\qquad \subset \Psi_0^{m+m',(\gamma_\lb,\gamma_\rb)}(X,\Omegazero^{\frac12}X) + \Psi_0^{-\infty,(\gamma_\lb,\gamma_\ff,\gamma_\rb)}(X,\Omegazero^{\frac12}X).
  \end{align*}
\end{prop}

\subsection{Schwartz kernel of the (reduced) normal operator}
\label{SsCN}

Let $x\geq 0$, $y\in\R^{n-1}$ denote local coordinates near a point in $\pa X$. The lifts of these functions along the left projection of $X^2_0$ and $X^2_{0'}$ to $X$ are denoted by the same letters $x,y$, and the lifts along the right projection by primed letters $x',y'$. (Discussions of the invariant content of various constructions below can be found in \cite{MazzeoMelroseHyp,LauterPsdoConfComp}.)

At first, let us work locally near $\ff\setminus\rb$, where we can use the coordinates
\begin{equation}
\label{EqCNCoord}
  s=\frac{x}{x'},\qquad
  \Ups=\frac{y-y'}{x'},\qquad
  x',\qquad
  y'.
\end{equation}
Consider two operators $P,Q\in\Psi_0^{-\infty}(X,\Omegazero^{\frac12}X)$ whose Schwartz kernels are supported away from $\lb\cup\rb$; note that the space of such operators is dense in the space $\Psi_0^{m,(\alpha_\lb,\alpha_\rb)}(X,\Omegazero^{\frac12}X)$ in the topology of $\Psi_0^{m+\eps,(\alpha_\lb-\eps,\alpha_\rb-\eps)}(X,\Omegazero^{\frac12}X)$ for any $m\in\R$, $\alpha_\lb,\alpha_\rb\in\R$, and $\eps>0$. Write the Schwartz kernel of $P$ in local coordinates as
\begin{equation}
\label{Eq0Op}
  K_P(x,y,x',y')\,\Bigl|\frac{\dd x}{x}\frac{\dd y}{x^{n-1}}\frac{\dd x'}{x'}\frac{\dd y'}{x'{}^{n-1}}\Bigr|^{\frac12} = K_P^0(s,\Ups,x',y')\,\Bigl|\frac{\dd x}{x}\frac{\dd y}{x^{n-1}}\frac{\dd x'}{x'}\frac{\dd y'}{x'{}^{n-1}}\Bigr|^{\frac12},
\end{equation}
and likewise for $Q$. The Schwartz kernel $K_{P\circ Q}^0(s,\Ups,x',y')|\frac{\dd x}{x}\frac{\dd y}{x^{n-1}}\frac{\dd x'}{x'}\frac{\dd y'}{x'{}^{n-1}}|^{\frac12}$ of $P\circ Q$ in the coordinates~\eqref{EqCNCoord} (thus $x=s x'$ and $y=y'+x'\Ups$) is then given by
\begin{align*}
  K_{P\circ Q}^0(s,\Ups,x',y') &= \iint K_P(s x',y'+x'\Ups,x'',y'')K_Q(x'',y'',x',y')\,\frac{\dd x''}{x''}\frac{\dd y''}{x''{}^{n-1}} \\
    &= \iint K_P^0\Bigl(\frac{s x'}{x''},\frac{y'-y''+x'\Ups}{x''},x'',y''\Bigr)K_Q^0\Bigl(\frac{x''}{x'},\frac{y''-y'}{x'},x',y'\Bigr)\,\frac{\dd x''}{x''}\frac{\dd y''}{x''{}^{n-1}} \\
    &= \iint K_P^0\Bigl(\frac{s}{s'},\frac{\Ups-\Ups'}{s'},s' x',y'+x'\Ups'\Bigr) K_Q^0(s',\Ups',x',y')\,\frac{\dd s'}{s'}\frac{\dd\Ups'}{s'{}^{n-1}},
\end{align*}
where we introduced $s'=\frac{x''}{x'}$ and $\Ups'=\frac{y''-y'}{x'}$. Upon restriction to the 0-front face $x'=0$, we thus find that
\[
  K_{P\circ Q}^0(s,\Ups,0,y') = \iint K_P^0\Bigl(\frac{s}{s'},\frac{\Ups-\Ups'}{s'},0,y'\Bigr)K_Q^0(s',\Ups',0,y')\,\frac{\dd s'}{s'}\frac{\dd\Ups'}{s'{}^{n-1}}.
\]
Thus, operator composition in the 0-calculus restricts at $\ff$ to the composition of convolution operators on the semidirect product $\R^{n-1}_\Ups\rtimes(\R_+)_s$, with smooth parametric dependence on the boundary point $y'$. That is, the normal operator map
\[
  N(P,y') := K_P^0|_{\ff_{y'}},
\]
where $\ff_{y'}$ is the fiber over $y'$ of the blow-down map $\ff\to\diag_{\pa X}\cong\pa X$, is a homomorphism from $\Psi_0^{-\infty}(X,\Omegazero^{\frac12}X)$ into the space of convolution operators on $\R^{n-1}\rtimes\R_+$. By the aforementioned density, the homomorphism property $N(P\circ Q,y')=N(P,y')\circ N(Q,y')$ continues to hold for elements $P,Q$ of the (extended) large calculus whenever the composition $P\circ Q$ is defined.

Let us extend the restriction $K_P^0(s,\Ups,0,y')$ to the Schwartz kernel of the unique operator which has normal operator $K_P^0(s,\Ups,0,y')$ and is invariant under the action of the group $\R^{n-1}\rtimes\R_+$. To make this explicit, we use tildes to denote coordinates on the model space $X_{y'}:=[0,\infty)_{\tilde x}\times\R^{n-1}_{\tilde y}$. Thus, we identify $N(P,y')$ with the operator in $\Psi_0^{-\infty}(X_{y'},\Omegazero^{\frac12}X_{y'})$ whose Schwartz kernel is
\[
  N(P,y')(\tilde x,\tilde y,\tilde x',\tilde y') = K_P^0\Bigl(\frac{\tilde x}{\tilde x'},\frac{\tilde y-\tilde y'}{\tilde x'},0,y'\Bigr)\,\Bigl|\frac{\dd\tilde x}{\tilde x}\frac{\dd\tilde y}{\tilde x^{n-1}}\frac{\dd\tilde x'}{\tilde x'}\frac{\dd\tilde y'}{\tilde x'{}^{n-1}}\Bigr|^{\frac12}.
\]
For fixed $(\tilde x,\tilde x')$, this is a convolution in the $\tilde y$-variables. Trivializing the $\half$-density bundle in the tangential variables $\tilde y,\tilde y'$ via $|\dd\tilde y\,\dd\tilde y'|$ and conjugating by the Fourier transform, we obtain an operator family with Schwartz kernel
\[
   (\tilde x,\tilde x') \mapsto \tilde x'{}^{n-1} \wh{K_P^0}\Bigl(\frac{\tilde x}{\tilde x'},\tilde x'\tilde\eta,0,y'\Bigr)\,\Bigl|\frac{\dd\tilde x}{\tilde x^n}\frac{\dd\tilde x'}{\tilde x'{}^n}\Bigr|^{\frac12},
\]
where $\wh{K_P^0}$ denotes the Fourier transform in the second argument ($\Ups$); this is an operator acting on the bundle of weighted b-$\half$-densities with local frame $|\frac{\dd\tilde x}{\tilde x^n}|^{\frac12}$. Multiplication by $\tilde x^{\frac{n-1}{2}}$ is an isomorphism between this bundle and the unweighted b-$\half$-density bundle, and upon conjugation by this weight we obtain the \emph{transformed normal operator} (changing $y'$ to $y_0$ for notational consistency with the introduction)
\begin{equation}
\label{EqCNOpTransf}
\begin{split}
  N(P,y_0,\tilde\eta)(\tilde x,\tilde x') &:= \tilde x^{\frac{n-1}{2}}\tilde x'{}^{-\frac{n-1}{2}}\cdot\tilde x'{}^{n-1} \wh{K_P^0}\Bigl(\frac{\tilde x}{\tilde x'},\tilde x'\tilde\eta,0,y'\Bigr)\,\Bigl|\frac{\dd\tilde x}{\tilde x^n}\frac{\dd\tilde x'}{\tilde x'{}^n}\Bigr|^{\frac12} \\
    &= \wh{K_P^0}\Bigl(\frac{\tilde x}{\tilde x'},\tilde x'\tilde\eta,0,y'\Bigr)\,\Bigl|\frac{\dd\tilde x}{\tilde x}\frac{\dd\tilde x'}{\tilde x'}\Bigr|^{\frac12},
\end{split}
\end{equation}
as an operator acting on sections of the b-$\half$-density bundle over $[0,\infty]$.

\begin{definition}[Normal operator]
\label{DefCN}
  In local coordinates $x,y$ on $X$, lifted to the left, resp.\ right factor of $X^2$ as $x,y$, resp.\ $x',y'$, write the Schwartz kernel of an element $P$ of the (large) (extended) 0-calculus as~\eqref{Eq0Op} in the coordinates~\eqref{EqCNCoord}. Then the operator $N(P,y_0,\tilde\eta)$ defined by~\eqref{EqCNOpTransf} (acting on b-$\half$-densities on $[0,\infty]$) is the \emph{transformed normal operator}. The \emph{reduced normal operator} is the operator family, acting on b-$\half$-densities on $[0,\infty]_t$, parameterized by $y_0\in\R^{n-1}$ and $\hat\eta\in\Sph^{n-2}$, with Schwartz kernel
  \[
    \hat N(P,y_0,\hat\eta)(t,t') := \wh{K_P^0}\Bigl(\frac{t}{t'},t'\hat\eta,0,y_0\Bigr)\,\Bigl|\frac{\dd t}{t}\frac{\dd t'}{t'}\Bigr|^{\frac12}.
  \]
\end{definition}

Thus, $\hat N(P,y_0,\hat\eta)$ arises from $N(P,y_0,\eta)$ by changing variables $t=\tilde x|\tilde\eta|$, $t'=\tilde x'|\tilde\eta|$ and setting $\hat\eta=\tilde\eta/|\tilde\eta|$. In the case that $P$ is the differential operator given by~\eqref{EqIOp} in the trivialization $|\frac{\dd x}{x}\frac{\dd y}{x^{n-1}}|^{\frac12}$ of $\Omegazero^{\frac12}X$, then $N(P,y_0,\eta)$ and $\hat N(P,y_0,\hat\eta)$ are (the Schwartz kernels of) the operators~\eqref{EqIOpNorm} and \eqref{EqIOpNormRed}. By construction, the maps $P\mapsto N(P,y_0,\eta)$ and $P\mapsto\hat N(P,y_0,\hat\eta)$ are homomorphisms in the generalized sense that they respect operator compositions in the large (extended) 0-calculus whenever the composition is defined. As a simple but important example, we note that the reduced normal operator of the identity $I$, $\hat N(I,y_0,\hat\eta)(t,t')=\delta(\frac{t}{t'})|\frac{\dd t}{t}\frac{\dd t'}{t'}|^{\frac12}$, is the identity operator on b-$\half$-densities.

\begin{rmk}[Short exact sequence]
\label{RmkCNSES}
  Directly from the definition, if all normal operators $N(P,y_0)$ of an operator $P\in\Psi_0^m(X,\Omegazero^{\frac12}X)$ vanish, then $P\in\rho_\ff\Psi_0^m(X,\Omegazero^{\frac12}X)$ (i.e.\ the Schwartz kernel vanishes to leading order at $\ff$). The analogous statements for the transformed or reduced normal operator are discussed in Remark~\ref{RmkCNDecReg}.
\end{rmk}

For the study of the boundary behavior of the reduced normal operator, it is computationally more convenient to pass to a different coordinate system which does not give preference to $\lb$ or $\rb$. Thus, consider as local coordinates near $\ff\setminus(\lb\cap\rb)=\ff'\setminus\ff_\bop$ inside the (extended) 0-double space of $X$
\begin{equation}
\label{EqCNGoodCoord}
  \rho=x+x',\qquad
  y' \in \R^{n-1},\qquad
  \tau:=\frac{x-x'}{x+x'} \in [-1,1],\qquad
  Y:=\frac{y-y'}{x+x'} \in \R^{n-1}.
\end{equation}
Thus, $\rho$ is a local defining function of $\ff$, while $\tau+1$ and $\tau-1$ are local defining functions of $\lb$ and $\rb$. Also, $\la Y\ra^{-1}$ extends by continuity to a local defining function of $\ff_\bop$ inside $X^2_{0'}$. We similarly define $\tilde\rho$, $\tilde y'$, $\tilde\tau$, $\tilde Y$ on $(X_{y'})_0^2$. Let us write (the Schwartz kernel of) $P$ in these coordinates as
\begin{equation}
\label{Eq0GoodSK}
  P = p(\rho,y',\tau,Y)\Bigl|\frac{\dd x}{x}\frac{\dd y}{x^{n-1}}\frac{\dd x'}{x'}\frac{\dd y'}{x'{}^{n-1}}\Bigr|^{\frac12}.
\end{equation}
(Thus, if $P$ lies in the small 0-calculus, then $p$ vanishes rapidly as $\tau\to\pm 1$ or $|Y|\to\infty$, and $p$ has a conormal singularity at the 0-diagonal $(\tau,Y)=(0,0)$.) The transformed normal operator is then
\begin{equation}
\label{EqCNGoodNOp}
\begin{split}
  N(P,y_0,\tilde\eta)(\tilde x,\tilde x') &:= \Bigl(\frac{\tilde x'}{\tilde x+\tilde x'}\Bigr)^{-(n-1)} \hat p\Bigl(y_0,\frac{\tilde x-\tilde x'}{\tilde x+\tilde x'},(\tilde x+\tilde x')\tilde\eta\Bigr) \Bigl|\frac{\dd\tilde x}{\tilde x}\frac{\dd\tilde x'}{\tilde x'}\Bigr|^{\frac12}, \\
  &\hspace{-3em} \hat p(y_0,\tau,\eta):=\int_{\R^{n-1}} e^{-i Y\cdot\eta}p(0,y_0,\tau,Y)\,\dd Y.
\end{split}
\end{equation}
The reduced normal operator is correspondingly (via $t=\tilde x|\tilde\eta|$, $t'=\tilde x'|\tilde\eta|$, $\hat\eta=\frac{\tilde\eta}{|\tilde\eta|}$ for $\tilde\eta\neq 0$, as before) given by
\begin{equation}
\label{EqCNGoodNOpResc}
  \hat N(P,y_0,\hat\eta)(t,t') = \Bigl(\frac{t'}{t+t'}\Bigr)^{-(n-1)} \hat p\Bigl(y_0,\frac{t-t'}{t+t'},(t+t')\hat\eta\Bigr) \Bigl|\frac{\dd t}{t}\frac{\dd t'}{t'}\Bigr|^{\frac12}.
\end{equation}
Conversely, we can recover $\hat p(y_0,\tau,\eta)$ for $\eta\neq 0$ from $\hat N(P,y_0,\hat\eta)$ by computing the solution $(t,t',\hat\eta)$ of the system $\frac{t-t'}{t+t'}=\tau$, $(t+t')\hat\eta=\eta$, which is given by $\hat\eta=\frac{\eta}{|\eta|}$, $(t,t')=\half|\eta|(1+\tau,1-\tau)$; this gives the formula
\begin{equation}
\label{EqCNRecover}
  \hat p(y_0,\tau,\eta) = \Bigl(\frac{1-\tau}{2}\Bigr)^{n-1}\biggl( \Bigl|\frac{\dd t}{t}\frac{\dd t'}{t'}\Bigr|^{-\frac12}\hat N\Bigl(P,y_0,\frac{\eta}{|\eta|}\Bigr)\biggr)\bigl(\half|\eta|(1+\tau),\half|\eta|(1-\tau)\bigr),\qquad\eta\neq 0.
\end{equation}

\subsection{b-scattering pseudodifferential operators}
\label{SsCNbsc}

Following \cite[\S3]{LauterPsdoConfComp}, we now describe the class of operators in which the reduced normal operator of a 0-ps.d.o.\ lies.

\begin{definition}[b-scattering double space]
\label{DefCNbsc}
  Let $[0,\infty]$ denote the compactification of $[0,\infty)_t$ with $(1+t)^{-1}$ as the defining function of infinity. We define the b-scattering double space of $[0,\infty]$ as
  \[
    [0,\infty]^2_{\bop,\scop} := \bigl[\,\ol{[0,\infty)^2}; \pa\diag \bigr],
  \]
  where $\ol{[0,\infty)^2}$ is the closure of the first quadrant in the radial compactification $\ol{\R^2}$ (thus it is closed quarter disc), and $\diag\subset\ol{[0,\infty)^2}$ is the closure of the diagonal $\{t=t'\}$, with boundary consisting of two points: $\pa\diag=\{(0,0),(\infty,\infty)\}$. We denote by $\diag_{\bop,\scop}\subset[0,\infty]^2_{\bop,\scop}$ the lifted diagonal, and by $\lb_\bop$, $\rb_\bop$, and $\ff_{\bop,0}$, $\ff_\scop$, and $\ff_{\bop,\infty}$ the lifts of $\{0\}\times[0,\infty]$, $[0,\infty]\times\{0\}$, and of $\{(0,0)\}$, $\{(\infty,\infty)\}$, and of the boundary at infinity of $\ol{[0,\infty)^2}$, respectively.
\end{definition}

See Figure~\ref{FigCNbsc}.

\begin{figure}[!ht]
\centering
\includegraphics{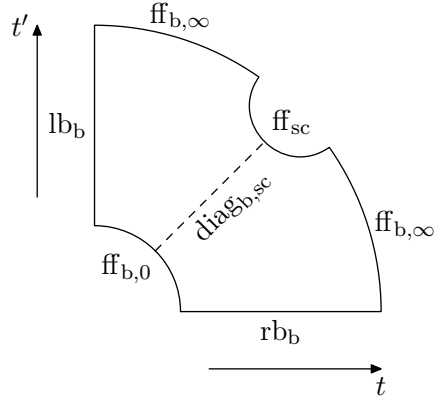}
\caption{The b-scattering double space.}
\label{FigCNbsc}
\end{figure}

In terms of the coordinates $\hat\tau=\frac{t-t'}{t+t'}$ and $\hat\rho=t+t'$ (analogous to $\tau$, $\rho$ in~\eqref{EqCNGoodCoord}), the b-scattering double space is diffeomorphic to $\bigl[ [-1,1]_{\hat\tau}\times[0,\infty]_{\hat\rho}; \{(0,\infty)\} \bigr]$.

\begin{definition}[b-scattering ps.d.o.s]
\label{DefCNbscPsdo}
  Let $\KD_\bop:=\pi_L^*\Omegab^{\frac12}[0,\infty]\otimes\pi_R^*\Omegab^{\frac12}[0,\infty]$ (which thus has the global section $|\frac{\dd t}{t}\frac{\dd t'}{t'}|^{\frac12}$). Let $\rho_{\ff_\scop}\in\CI([0,\infty]_{\bop,\scop}^2)$ denote a defining function of $\ff_\scop$. For $m,r\in\R$, we define\footnote{In view of the rapid vanishing at $\ff_{\bop,\infty}$, one can equivalently replace $\rho_{\ff_\scop}$ in this definition by the total defining function $(1+t+t')^{-1}$ of $\ff_\scop\cup\ff_{\bop,\infty}$.} \footnote{We write $\Omegab^{\frac12}=\Omegab^{\frac12}[0,\infty]$ here for better readability.}
  \[
    \Psi_{\bop,\scop}^{m,(0,r)}([0,\infty],\Omegab^{\frac12}) := \bigl\{ \kappa\in \cA_{\ff_\scop}^{-r-1} I^m([0,\infty]^2_{\bop,\scop}, \diag_{\bop,\scop}; \KD_\bop) \colon \kappa\equiv 0\ \text{at}\ \lb_\bop\cup\rb_\bop\cup\ff_{\bop,\infty} \bigr\},
  \]
  where $\cA_{\ff_\scop}^{-r-1}I^m$ is the space of conormal distributions which are smooth down to $\ff_\bop$ and conormal down to $\ff_\scop$.\footnote{That is, omitting the density bundle, they are inverse Fourier transforms from the $\sigma$-variable to $t-t'$ of symbols $a=a(t,\sigma)$ on $[0,\infty]_t\times\R_\sigma$ satisfying estimates $|\pa_t^j\pa_\sigma^k a|\lesssim\la\sigma\ra^{m-k}$ for $t\leq 2$ and $|(t\pa_t)^j\pa_\sigma^k a|\lesssim (t^{-1})^{-r-1}\la\sigma\ra^{m-k}$ for $t\geq 1$.} For a collection of index sets $\cE=(\cE_{\lb_\bop},\cE_{\ff_{\bop,0}},\cE_{\rb_\bop})$, we define
  \[
    \Psi_{\bop,\scop}^{m,\cE}([0,\infty]),\Omegab^{\frac12}) := \bigl\{ \kappa\in \cA_\phg^\cE I^m([0,\infty]^2_{\bop,\scop}, \diag_{\bop,\scop}; \KD_\bop) \colon \kappa\equiv 0\ \text{at}\ \ff_{\bop,\infty}\cup\ff_\scop \bigr\},
  \]
  i.e.\ the index sets at $\ff_\scop$ and $\ff_{\bop,\infty}$ are trivial.\footnote{For $m=-\infty$, this space equals $\cA_\phg^{(\cE_{\lb_\bop},\cE_{\ff_{\bop,0}},\cE_{\rb_\bop},\emptyset,\emptyset)}([0,\infty]_{\bop,\scop}^2,\KD_\bop)$.} For $\cE'=(\cE_0,\cE_1)$, we finally set
  \[
    \Psi^{-\infty,\cE'}([0,\infty],\Omegab^{\frac12}) := \cA_\phg^{(\cE_0,\cE_1,\emptyset)}\bigl(\ol{[0,\infty)^2},\KD_\bop\bigr),
  \]
  with index set $\cE_0$, $\cE_1$, $\emptyset$ assigned to $\{0\}\times[0,\infty]$, $[0,\infty]\times\{0\}$, and the boundary at infinity, respectively.
\end{definition}

The shift of the weight $-r$ at $\ff_\scop$ by $-1$ is due to the relationship between scattering and b-$\half$-density bundles near $\ff_\scop$
\[
  |\dd t\,\dd t'|^{\frac12} = t^{\frac12}t'{}^{\frac12} \Bigl|\frac{\dd t}{t}\frac{\dd t'}{t'}\Bigr|^{\frac12},
\]
with $t^{\frac12}t'{}^{\frac12}$ a smooth positive multiple of $\rho_{\ff_\scop}^{-1}\rho_{\ff_{\bop,\infty}}^{-1}$ near $\ff_\scop$. The chosen normalization thus ensures that the identity operator on b-$\half$-densities on $[0,\infty]$ lies in $\Psi_{\bop,\scop}^{0,(0,0)}([0,\infty],\Omegab^{\frac12})$.

Under composition, b-scattering-ps.d.o.s behave like b-ps.d.o.s near $\lb_\bop\cup\ff_{\bop,0}\cup\rb_\bop$ and like scattering ps.d.o.s near $\ff_\scop$. We refer the reader to \cite[\S3]{LauterPsdoConfComp}, \cite[\S5]{MelroseAPS}, \cite{MelroseEuclideanSpectralTheory} for more in-depth treatments. Here, we only record:

\begin{prop}[Composition of b-scattering-ps.d.o.s]
\label{PropCNbsc}
  Let $\cE=(\cE_{\lb_\bop},\cE_{\ff_{\bop,0}},\cE_{\rb_\bop})$ and $\cF=(\cF_{\lb_\bop},\cF_{\ff_{\bop,0}},\cF_{\rb_\bop})$ be two index sets. Define $\cG=(\cG_{\lb_\bop},\cG_{\ff_{\bop,0}},\cG_{\rb_\bop})$ by
  \begin{gather*}
    \cG_{\lb_\bop} = \cE_{\lb_\bop} \extcup\,(\cE_{\ff_{\bop,0}}+\cF_{\lb_\bop}), \qquad
    \cG_{\rb_\bop} = (\cE_{\rb_\bop} + \cF_{\ff_{\bop,0}}) \extcup \cF_{\rb_\bop}, \\
    \cG_{\ff_{\bop,0}} = (\cE_{\ff_{\bop,0}}+\cF_{\ff_{\bop,0}}) \extcup\,(\cE_{\lb_\bop}+\cF_{\rb_\bop}).
  \end{gather*}
  The composition of b-scattering ps.d.o.s (acting on b-$\half$-densities on $[0,\infty]$) then satisfies:
  \begin{alignat*}{2}
    \Psi_{\bop,\scop}^{m,(0,r)} &\circ \Psi_{\bop,\scop}^{m',(0,r')} &\subset& \Psi_{\bop,\scop}^{m+m',(0,r+r')}, \\
    \Psi_{\bop,\scop}^{m,(0,r)} &\circ \Psi_{\bop,\scop}^{-\infty,\cE} &\subset& \Psi_{\bop,\scop}^{-\infty,\cE}, \\
    \Psi_{\bop,\scop}^{-\infty,\cE} &\circ \Psi_{\bop,\scop}^{-\infty,\cF} &\subset& \Psi_{\bop,\scop}^{-\infty,\cG}.
  \end{alignat*}
\end{prop}

\subsection{Range of the reduced normal operator}

The range of the homomorphism from 0-ps.d.o.s to reduced normal operators is difficult to describe in a useful manner; this means that special care is required in ensuring that parametrix constructions on the level of the reduced normal operator remain in the range of the reduced normal operator. In this section, we collect several results which will facilitate this task in~\S\ref{SP}.

\begin{prop}[Reduced normal operator for the small 0-calculus]
\label{PropCNSmall}
  Fix an operator $P\in\Psi_0^m(X,\Omegazero^{\frac12}X)$. Then $\hat N(P,y_0,\hat\eta)\in\Psi_{\bop,\scop}^{m,(0,m)}([0,\infty],\Omegab^{\frac12})$, with smooth parametric dependence on $y_0\in\R^{n-1}$ and $\hat\eta\in\Sph^{n-2}$. Moreover, if $\sigmazero^m(P)|_{\Tzero^*_{y_0}X}$ is elliptic, then $\hat N(P,y_0,\hat\eta)$ is elliptic, in the sense that its principal symbol
  \begin{equation}
  \label{EqCNSmallSymbol}
    {}^{\bop,\scop}\sigma^{m,(0,m)}(\hat N(P,y_0,\hat\eta)) \in (S^{m,(0,m)}/S^{m-1,(0,m-1)})({}^{\bop,\scop}T^*[0,\infty])
  \end{equation}
  is invertible.\footnote{In particular, in $T=t^{-1}<1$, the operator $\hat N(P,y_0,\hat\eta)$ is fully elliptic as a scattering ps.d.o.} Finally, the indicial operator $I(P,y_0)$ of $\hat N(P,y_0,\hat\eta)$ at $t=0$ is independent of $\hat\eta\in\Sph^{n-2}$.
\end{prop}

Here, ${}^{\bop,\scop}T^*[0,\infty]$ has local frame $\frac{\dd t}{t}$ in $t\leq 2$ and $\dd t=-\frac{\dd T}{T^2}$ in $T=t^{-1}\leq 1$. Moreover, $S^{m,(0,r)}({}^{\bop,\scop}T^*[0,\infty])$ is the space of symbols of order $m$ which are smooth down to $t=0$ and conormal with weight $(t^{-1})^{-r}$ at $t^{-1}=0$. We also used the following terminology:

\begin{definition}[Indicial operator]
\label{DefCNbnorm}
  The \emph{indicial operator} (or \emph{b-normal operator}) $I(P,y_0)$ of $\hat N(P,y_0,\hat\eta)$ is the dilation-invariant extension of the restriction $N(P,y_0,\hat\eta)|_{\ff_{\bop,0}}$ to a b-ps.d.o.\ acting on $\half$-b-densities on $[0,\infty]$; thus, its Schwartz kernel is
  \begin{equation}
  \label{EqCNbNormal}
    I(P,y_0)(t,t') = \Bigl(\frac{t'}{t+t'}\Bigr)^{-(n-1)}\hat p\Bigl(y_0,\frac{t-t'}{t+t'},0\Bigr)\,\Bigl|\frac{\dd t}{t}\frac{\dd t'}{t'}\Bigr|^{\frac12} = \wh{K_P^0}\Bigl(\frac{t}{t'},0,0,y_0\Bigr)\,\Bigl|\frac{\dd t}{t}\frac{\dd t'}{t'}\Bigr|^{\frac12}
  \end{equation}
  in the notation of~\eqref{Eq0GoodSK}, \eqref{EqCNGoodNOp} and~\eqref{Eq0Op}, \eqref{EqCNOpTransf}.
\end{definition}

\begin{proof}[Proof of Proposition~\usref{PropCNSmall}]
  Write $P$ as in~\eqref{Eq0GoodSK}; we restrict $p$ to the 0-front face $\rho=0$ and fix $y'=y_0$. Thus, $p|_{\ff_{y_0}}=p(0,y_0,\tau,Y)$ is an oscillatory integral
  \[
    p(0,y_0,\tau,Y) = (2\pi)^{-n}\iint e^{i(\tau\xi+Y\cdot\eta)} a(\xi,\eta)\,\dd\xi\,\dd\eta
  \]
  for some symbol $a\in S^m(\R^n_{(\xi,\eta)})$; moreover, $p$ is smooth away from $(\tau,Y)=(0,0)$ and vanishes to infinite order as $\tau\to\pm 1$ or $|Y|\to\infty$. We then have
  \begin{equation}
  \label{EqCNSmallFT}
    \hat p(y_0,\tau,\eta) = (2\pi)^{-1}\int e^{i\tau\xi}a(\xi,\eta)\,\dd\xi.
  \end{equation}
  This is a smooth function of $\eta$ with values in distributions on $(-1,1)_\tau$ which are conormal (of order $m$) to $0$ and vanish to infinite order at $\tau=\pm 1$. The formula~\eqref{EqCNGoodNOpResc} thus shows that $\hat N(P,y_0,\hat\eta)$ lies in $\Psib^m([0,\infty)_t)$. Note also that if $(\xi,\eta)\mapsto a(\xi,\eta)$ is an elliptic symbol, then also $\xi\mapsto a(\xi,\eta)$ is elliptic for any fixed $\eta\in\R^{n-1}$.

  Near $\ff_\scop^\circ\subset[0,\infty]_{\bop,\scop}^2$, we pass to the coordinates
  \[
    T = (t+t')^{-1},\qquad
    s = \frac{t-t'}{t+t'} \Big/ T = t-t',
  \]
  i.e.\ $t=\half(T^{-1}+s)$ and $t'=\half(T^{-1}-s)$. Thus, $T$ is a local defining function of $\ff_\scop^\circ$, and $\diag_{\bop,\scop}=\{s=0\}$. We then compute the Schwartz kernel of $T^{m+1}\hat N(P,y_0,\hat\eta)$ to be
  \[
    (T,s) \mapsto \bigl(\half(1-T s)\bigr)^{-(n-1)} T^{m+1} \hat p(y_0,s T, T^{-1}\hat\eta)\,\Bigl|\frac{\dd t}{t}\frac{\dd t'}{t'}\Bigr|^{\frac12}.
  \]
  But changing variables via $T\xi=\sigma$ in~\eqref{EqCNSmallFT}, we obtain
  \begin{align*}
    T^{m+1}\hat p(y_0,s T,T^{-1}\hat\eta) &= (2\pi)^{-1}\int e^{i s T \xi} T^{m+1}a(\xi,T^{-1}\hat\eta)\,\dd\xi \\
      &= (2\pi)^{-1} \int e^{i s\sigma} T^m a(T^{-1}\sigma,T^{-1}\hat\eta)\,\dd\sigma.
  \end{align*}
  The map $a_\scop(\hat\eta;T,\sigma):=T^m a(T^{-1}\sigma,T^{-1}\hat\eta)$ obeys the conormal (at $T=0$) symbolic (in $\sigma$) bounds $|(T\pa_T)^j \pa_\sigma^k a_\scop| \lesssim \la\sigma\ra^{m-k-|\alpha|}$, as do its derivatives in $\hat\eta$. We have thus proved that $\hat N(P,y_0,\hat\eta)\in\Psi_{\bop,\scop}^{m,(0,m)}([0,\infty],\Omegab^{\frac12})$, with smooth dependence on $y_0$, $\hat\eta$. (The rapid vanishing at $\ff_{\bop,\infty}=s^{-1}(\{\pm\infty\})$ is a consequence of the fact that the inverse Fourier transform of a symbol vanishes rapidly at infinity.)
  
  Moreover, if $a(\sigma,\eta)$ is elliptic, i.e.\ is bounded from below by $c(1+|\sigma|+|\eta|)^m$ for $|\sigma|+|\eta|>C$ for some constants $c,C>0$, then also
  \[
    |a_\scop(\hat\eta;T,\sigma)| \geq c T^m (1 + T^{-1}|\sigma| + T^{-1})^m \geq c(1+|\sigma|)^m
  \]
  provided $|\sigma|+1>C T$, and thus for all $\sigma\in\R$ when $T<C^{-1}$. Together with the ellipticity of $\hat N(P,y_0,\hat\eta)$ near $t=0$, this proves the ellipticity of~\eqref{EqCNSmallSymbol}.
\end{proof}

Recall that parametrix constructions in the b-calculus require the inversion of the b-normal operator, which in the present context concretely concerns the indicial operator $I(P,y_0)$ from Definition~\ref{DefCNbnorm}. The \emph{indicial family} for a fixed boundary point $y_0$ is the Mellin transform of $I(P,y_0)|_{\ff_{\bop,0}}$ (which can be identified with a b-$\half$-density on $\ff_{\bop,0}$, see \cite[\S4.15]{MelroseAPS}) in the projective coordinate $s=t/t'\in[0,\infty]$ along $\ff_{\bop,0}$, and hence (upon dropping the $\half$-density factor $|\dd\sigma|^{\frac12}$) explicitly given by
\[
  I(P,y_0,\sigma) := \int_0^\infty s^{i\sigma}\wh{K_P^0}(s,0,0,y_0)\,\frac{\dd s}{s}.
\]

\begin{lemma}[Properties of the indicial family]
\label{LemmaCNbsc2Ind}
  Let $P\in\Psi_0^m(X,\Omegazero^{\frac12}X)$. Then the indicial family $I(P,y_0,\sigma)$ is holomorphic in $\sigma\in\C$. Furthermore, if $P$ is elliptic, then for any $C_1>0$ there exist $c,C_2>0$ so that
  \begin{equation}
  \label{EqCNbsc2Ind}
    |I(P,y_0,\sigma)|>c|\sigma|^m,\qquad |\Im\sigma|<C_1,\ |\Re\sigma|>C_2.
  \end{equation}
\end{lemma}
\begin{proof}
  The function $s\mapsto\wh{K_P^0}(s,0,0,y_0)$ vanishes to infinite order at $s=0$ and $s=\infty$, which implies the holomorphicity of $I(P,y_0,\sigma)$. When $P$ is elliptic, we write
  \[
    K_P^0(s,Y,0,y_0) = (2\pi)^{-n}\iint e^{i(\sigma\log s+\eta\cdot Y)}a(y_0,\sigma,\eta)\,\dd\sigma\,\dd\eta
  \]
  where $a$ is an elliptic symbol of order $m$. Thus, $\wh{K_P^0}(s,0,0,y_0)=(2\pi)^{-1}\int s^{i\sigma} a(y_0,\sigma,0)\,\dd\sigma$ has Mellin transform $I(P,y_0,\sigma)=a(y_0,\sigma,0)$ which thus satisfies~\eqref{EqCNbsc2Ind} for real $\sigma$. For $\Im\sigma=\alpha$ with arbitrary $\alpha\in\R$, the estimate~\eqref{EqCNbsc2Ind} follows from the fact that also $s^\alpha K_P^0(s,Y,0,y_0)$ is the inverse Fourier transform of an elliptic symbol.
\end{proof}

The boundary spectrum
\[
  \Specb(P,y_0) \subset \C\times\N_0
\]
is defined by~\eqref{EqISpecb}.

\begin{cor}[Index set from boundary spectrum]
\label{CorCNInd}
  Suppose $P\in\Psi_0^m(X,\Omegazero^{\frac12}X)$ is elliptic, and let $\alpha\in\R$. Let $\cE\subset\C\times\N_0$ the smallest set containing $\{(z,k)\in\Specb(P,y_0)\colon\Re z>\alpha\}$ which has the properties that $(z,k)\in\cE$ implies $(z+1,k)\in\cE$ and (when $k\geq 1$) $(z,k-1)\in\cE$. Then $\cE$ is an index set (as defined before Definition~\usref{DefIFull}). The same is true if instead $\cE$ is required to contain $\{(-z,k)\colon(z,k)\in\Specb(P,y_0),\ \Re z<\alpha\}$.
\end{cor}
\begin{proof}
  Lemma~\ref{LemmaCNbsc2Ind} implies that for any $C$, the set $\{(z,k)\in\Specb(P,y_0)\colon\alpha<\Re z<C\}$ is finite, as is the set $\{(z,k)\in\Specb(P,y_0)\colon-C<\Re z<\alpha\}$. This implies the claim.
\end{proof}

For special classes of residual operators, it is possible to give a full characterization of the range of the reduced normal operator map:

\begin{lemma}[Reduced normal operator for residual extended 0-ps.d.o.s]
\label{LemmaCNResidual}
  Let $\cE_{\lb'},\cE_{\rb'}\subset\C\times\N_0$ denote two index sets, and put $\cE=(\cE_{\lb'},\N_0,\emptyset,\cE_{\rb'})$ (i.e.\ the index set at $\ff_\bop$ is trivial). Then for $P\in\Psi_{0'}^{-\infty,\cE}(X,\Omegazero^{\frac12}X)$, we have
  \[
    \hat N(P,y_0,\hat\eta) \in \CI\bigl(\Sph^{n-2}_{\hat\eta};\Psi_{\bop,\scop}^{(\cE_{\lb'},\N_0,\cE_{\rb'}-(n-1))}([0,\infty],\Omegab^{\frac12})\bigr),
  \]
  and in fact the function $\hat p(y_0,\tau,\eta)$ defined in terms of $\hat N(P,y_0,\hat\eta)$ for $\eta\neq 0$ by~\eqref{EqCNRecover} extends across $\eta=0$ to an element
  \begin{equation}
  \label{EqCNResidualSchwartz}
    \hat p(y_0,\cdot,\cdot) \in \sS\bigl(\R^{n-1}_\eta; \cA_\phg^{(\cE_{\lb'},\cE_{\rb'})}([-1,1]_\tau)\bigr).
  \end{equation}
  Conversely, any family of operators $N\in\CI\bigl(\Sph^{n-2}_{\hat\eta};\Psi_{\bop,\scop}^{(\cE_{\lb'},\N_0,\cE_{\rb'}-(n-1))}\bigr)([0,\infty],\Omegab^{\frac12}))$ with the property that $(-1,1)\times(\R^{n-1}\setminus\{0\})\ni(\tau,\eta)\mapsto(\frac{1-\tau}{2})^{n-1}(|\frac{\dd t}{t}\frac{\dd t'}{t'}|^{-\frac12}N(\frac{\eta}{|\eta|}))(\half|\eta|(1+\tau),\half|\eta|(1-\tau))$ extends to an element of the space in~\eqref{EqCNResidualSchwartz} is the reduced normal operator of an element of $\Psi_{0'}^{-\infty,\cE}(X,\Omegazero^{\frac12}X)$.
\end{lemma}
\begin{proof}
  This follows directly from the definition~\eqref{EqCNGoodNOpResc} of $\hat N(P,y_0,\hat\eta)$ (for the first part) and the formula~\eqref{EqCNRecover} (for the second part). The shift of the index set at $\rb_\bop$ comes from the factor $(\frac{t'}{t+t'})^{-(n-1)}$; note here that $\frac{t'}{t+t'}$ is a defining function of $\rb_\bop$.
\end{proof}

\begin{rmk}[Decay at $\ff_\bop$ and regularity in $\eta$]
\label{RmkCNDecReg}
  When the Schwartz kernel of an element $P\in\Psi_{0'}^{-\infty,(\cE_{\lb'},\N_0,\cE_{\ff_\bop},\cE_{\rb'})}(X,\Omegazero^{\frac12}X)$ has nontrivial index set $\cE_{\ff_\bop}$ at $\ff_\bop$, then $\hat p(y_0,\tau,\eta)$ is typically singular (albeit conormal) at $\eta=0$. Since the Schwartz kernel of $\hat N(P,y_0,\hat\eta)$ is defined only with reference to $\eta\neq 0$, possible (differentiated) $\delta$-distributional contributions to $\hat p(y_0,\tau,\eta)$ at $\eta=0$ are lost when passing to the transformed or reduced normal operator. This problem does not occur however when $\Re\cE_{\ff_\bop}>0$, and this will always be the case in this paper; thus, under this condition, if all reduced normal operators of $P$ are zero, then $P$ vanishes to leading order at $\ff'$, i.e.\ $P\in\Psi_{0'}^{-\infty,(\cE_{\lb'},\N_0+1,\cE_{\ff_\bop},\cE_{\rb'})}(X,\Omegazero^{\frac12}X)$. For $P\in\Psi_0^{-\infty,(\cE_\lb,\N_0,\cE_\rb)}(X,\Omegazero^{\frac12}X)$, the corresponding condition is $\Re(\cE_\lb+\cE_\rb)>0$ by~\eqref{EqCRel}.
\end{rmk}

\begin{cor}[0-operator from indicial operator]
\label{CorCN0fromInd}
  Let $\cE_0,\cE_1\subset\C\times\N_0$ be two index sets. Suppose $a\in\cA_\phg^{(\cE_0,\cE_1)}([-1,1])$. Then there exists $P\in\Psi_{0'}^{-\infty,(\cE_0,\N_0,\emptyset,\cE_1+(n-1))}(X,\Omegazero^{\frac12}X)$ with $I(P,y_0)(t,t')=a(\frac{t-t'}{t+t'})|\frac{\dd t}{t}\frac{\dd t'}{t'}|^{\frac12}$. If $a$ in addition depends smoothly on a parameter $y\in\R^{n-1}$, so $a=a(y)$, then one can find a single such $P$ with $I(P,y)=a(y)$ for all $y$.
\end{cor}
\begin{proof}
  Let $\chi\in\CIc(\R^{n-1})$ be identically $1$ near $0$. Set
  \[
    q(\tau,\eta) = \Bigl(\frac{1-\tau}{2}\Bigr)^{n-1}a(\tau)\chi(\eta),
  \]
  which lies in the space~\eqref{EqCNResidualSchwartz} for $\cE_{\lb'}=\cE_0$, $\cE_{\rb'}=\cE_1+(n-1)$. Thus, by (the proof of) Lemma~\ref{LemmaCNResidual}, if we set
  \[
    p(y_0,\tau,Y) := (2\pi)^{-(n-1)}\int_{\R^{n-1}} e^{i Y\cdot\eta}q(\tau,\eta)\,\dd\eta,
  \]
  then $p(y_0,\tau,Y)|\frac{\dd x}{x}\frac{\dd y}{x^{n-1}}\frac{\dd x'}{x'}\frac{\dd y'}{x'{}^{n-1}}|^{\frac12}$ is the restriction of the Schwartz kernel of an element $P\in\Psi_{0'}^{-\infty,(\cE_0,\N_0,\emptyset,\cE_1+(n-1))}(X,\Omegazero^{\frac12}X)$ to $\ff_{y_0}$. By construction, the reduced normal operator satisfies $N(P,y_0,0)(t,t')=(\frac{1-\tau}{2})^{-(n-1)}q(\tau,0)=a(\tau)$ where $\tau=\frac{t-t'}{t+t'}$.

  Since the given construction depends smoothly on the point $y_0$, the final statement of the Corollary follows.
\end{proof}

The next result shows the surjectivity of the reduced normal operator map onto fully residual operators:

\begin{prop}[0-ps.d.o.\ from fully residual reduced normal operators]
\label{PropCNConstr}
  Consider an operator family $N(\hat\eta)\in\CI(\Sph^{n-2};\Psi^{-\infty,(\cE_0,\cE_1)}([0,\infty],\Omegab^{\frac12}))$ with $\Re(\cE_0+\cE_1)>-(n-1)$. Let $y_0\in\R^{n-1}$ denote a point on $\pa X$. Then there exists $P\in\Psi_0^{-\infty,(\cE_0,\N_0,\cE_1+(n-1))}(X,\Omegazero^{\frac12}X)$ with $N(P,y_0,\hat\eta)=N(\hat\eta)$ for all $\hat\eta\in\Sph^{n-2}$. If $N$ in addition depends smoothly on a parameter $y\in\R^{n-1}$, so $N=N(y,\hat\eta)$, then one can find $P$ with $N(P,y,\hat\eta)=N(y,\hat\eta)$ for all $y,\hat\eta$.
  
  The analogous conclusions hold when $N(\hat\eta)$ takes values in $\Psi^{-\infty,(\alpha_0,\alpha_1)}([0,\infty],\Omegab^{\frac12})$ for $\alpha_0+\alpha_1>-(n-1)$, with $P\in\Psi_0^{-\infty,(\alpha_0,\alpha_1+(n-1))}(X,\Omegazero^{\frac12}X)$ then.
\end{prop}
\begin{proof}
  Write the Schwartz kernel of $N(\hat\eta)$ as
  \[
    n(\hat\eta;t,t')\Bigl|\frac{\dd t}{t}\frac{\dd t'}{t'}\Bigr|^{\frac12}.
  \]
  Let $\alpha_i<\Re\cE_i$ for $j=0,1$ be such that $\alpha_0+\alpha_1>-(n-1)$. We then have estimates
  \begin{equation}
  \label{EqCNConstrCon}
    |\pa_{\hat\eta}^\alpha(t\pa_t)^j(t'\pa_{t'})^k n| \lesssim t^{\alpha_0}t'{}^{\alpha_1} (1+t+t')^{-N}
  \end{equation}
  for all $\alpha,j,k,N$. The polyhomogeneity of $n$ is equivalent to the statement that for all $C\in\R$,
  \begin{equation}
  \label{EqCNConstrPhg}
    \biggl|\pa_{\hat\eta}^\alpha\biggl(\prod_{\genfrac{}{}{0pt}{}{(z,k)\in\cE_0}{\Re z\leq C}} (t\pa_t-z)\biggr) (t'\pa_{t'})^k n\biggr| \lesssim  t^C t'{}^{\alpha_1} (1+t+t')^{-N},
  \end{equation}
  together with analogous estimates capturing the expansion at $t'=0$; see \cite[Proposition~4.14.2]{MelroseDiffOnMwc}.

  By analogy with~\eqref{EqCNRecover}, define now for $\tau\in(-1,1)$ and $\eta\in\R^{n-1}\setminus\{0\}$
  \begin{equation}
  \label{EqCNConstrqDef}
    q(\tau,\eta) := \Bigl(\frac{1-\tau}{2}\Bigr)^{n-1} n\Bigl(\frac{\eta}{|\eta|},\half|\eta|(1+\tau),\half|\eta|(1-\tau)\Bigr).
  \end{equation}
  The estimates~\eqref{EqCNConstrCon} imply
  \begin{equation}
  \label{EqCNConstrq}
    \bigl|\bigl((1-\tau^2)\pa_\tau\bigr)^j \eta^\alpha\pa_\eta^\beta q(\tau,\eta)\bigr|\lesssim(1+\tau)^{\alpha_0}(1-\tau)^{\alpha_1+n-1}|\eta|^{\alpha_0+\alpha_1}(1+|\eta|)^{-N}
  \end{equation}
  for all $j$ and $\alpha,\beta\in\N_0^{n-1}$ with $|\beta|\leq|\alpha|$. Thus, the function $q$ can be extended uniquely across $\eta=0$ as an element of $L^1_\loc((-1,1)_\tau\times\R^{n-1}_\eta)$ (which is conormal at $\eta=0$), with rapid decay as $|\eta|\to\infty$ for any fixed $\tau\in(-1,1)$. As such, we can take its inverse Fourier transform in $\eta$,
  \begin{equation}
  \label{EqCNConstrpDef}
    p(\tau,Y) := (2\pi)^{-(n-1)} \int_{\R^{n-1}} e^{i Y\cdot\eta}q(\tau,\eta)\,\dd\eta.
  \end{equation}
  The bound~\eqref{EqCNConstrq} implies $|p(\tau,Y)| \lesssim(1+\tau)^{\alpha_0}(1-\tau)^{\alpha_1+n-1}$, and similarly for derivatives along any power of $(1-\tau^2)\pa_\tau$ and along $Y^\beta\pa_Y^\alpha$ for $|\beta|\leq|\alpha|$. We can prove a better bound for $|Y|\geq 1$ by exploiting the conormal regularity of $q$ at $\eta$: we split the domain of integration and prepare for an integration by parts argument (for improved decay in $Y$) by writing
  \[
    |p(\tau,Y)| \lesssim \int_{|\eta|<|Y|^{-1}} |q(\tau,\eta)|\,\dd\eta + \biggl|\int_{|\eta|>|Y|^{-1}} \Bigl(\bigl( |Y|^{-2}Y\cdot D_\eta\bigr)^M e^{i Y\cdot\eta}\Bigr) q(\tau,\eta)\,\dd\eta \biggr|
  \]
  for some $M$ chosen below. Write $w:=(1+\tau)^{\alpha_0}(1-\tau)^{\alpha_1+n-1}$, then the first integral is bounded by
  \[
    w\int_0^{|Y|^{-1}} r^{\alpha_0+\alpha_1}r^{n-2}\,\dd r\lesssim w|Y|^{-(n-1)-\alpha_0-\alpha_1}.
  \]
  In the second integral, we can integrate by parts $M$ times; each boundary term at $|\eta|=|Y|^{-1}$ can be estimated by $w|Y|^{-k-1}\int_{|\eta|=|Y|^{-1}}|\eta|^{\alpha_0+\alpha_1-k}\,\dd\eta\sim w|Y|^{-\alpha_0-\alpha_1-(n-1)}$ for some $k=0,\ldots,M-1$, and the final bulk integral is bounded from above by
  \begin{align*}
    w|Y|^{-M}\int_{|\eta|>|Y|^{-1}} |\eta|^{\alpha_0+\alpha_1-M}(1+|\eta|)^{-N}\,\dd\eta &\lesssim w|Y|^{-M}\int_{|Y|^{-1}}^\infty r^{\alpha_0+\alpha_1-M+n-2}\,\dd r \\
      &\lesssim w|Y|^{-\alpha_0-\alpha_1-n+1}
  \end{align*}
  as well, provided we choose $M>\alpha_0+\alpha_1+n-1$. Altogether, we have thus proved
  \begin{equation}
  \label{EqCNConstrEst}
    |p(\tau,Y)| \lesssim \Bigl(\frac{1+\tau}{\la Y\ra}\Bigr)^{\alpha_0}\Bigl(\frac{1-\tau}{\la Y\ra}\Bigr)^{\alpha_1+n-1},
  \end{equation}
  likewise for derivatives along $(1-\tau^2)\pa_\tau$ and $Y^\beta\pa_Y^\alpha$, $|\beta|\leq|\alpha|$.

  This gives the desired conormality of $p$ at $\lb$ and $\rb$ away from $\lb\cap\rb$. To prove that indeed $p\in\cA^{(\alpha_0,\alpha_1+n-1)}(\ff)$, $\ff\subset X^2_0$, it remains to analyze the behavior of $p$ near the corner $\lb\cap\rb\subset X^2_0$, where we pass to the coordinates
  \[
    \rho_\lb = \frac{1+\tau}{|Y|},\qquad
    \rho_\rb = \frac{1-\tau}{|Y|},\qquad
    \hat Y = \frac{Y}{|Y|};
  \]
  these are related to the coordinates $\tau$, $\rho_Y:=|Y|^{-1}$, $\hat Y$ near $\ff'\cap\ff_\bop\subset X^2_{0'}$ by $\rho_Y=\frac{\rho_\lb+\rho_\rb}{2}$ and $\tau=\frac{\rho_\lb-\rho_\rb}{\rho_\lb+\rho_\rb}$. While~\eqref{EqCNConstrEst} thus already gives the desired $L^\infty$-bound for $p$, it is more transparent to write
  \[
    p_0(\rho_\lb,\rho_\rb,\hat Y)=p\Bigl(\frac{\rho_\lb-\rho_\rb}{\rho_\lb+\rho_\rb},\frac{2\hat Y}{\rho_\lb+\rho_\rb}\Bigr)
  \]
  for the function $p$ in the new coordinates; using~\eqref{EqCNConstrqDef} and \eqref{EqCNConstrpDef} and the change of variables $\eta=(\rho_\lb+\rho_\rb)\zeta$, one finds
  \begin{equation}
  \label{EqCNConstrp0}
    p_0(\rho_\lb,\rho_\rb,\hat Y) = (2\pi)^{-(n-1)}\int_{\R^{n-1}} e^{2 i\hat Y\cdot\zeta} \rho_\rb^{n-1} n\Bigl(\frac{\zeta}{|\zeta|},\rho_\lb|\zeta|,\rho_\rb|\zeta|\Bigr)\,\dd\zeta,
  \end{equation}
  with the pointwise bound
  \begin{equation}
  \label{EqCNConstrp0Bd}
    |p_0|\lesssim\rho_\lb^{\alpha_0}\rho_\rb^{\alpha_1+n-1}
  \end{equation}
  from~\eqref{EqCNConstrEst}. Directly differentiating the integral expression~\eqref{EqCNConstrp0} shows that the derivative of $p_0$ along any power of $\rho_\lb\pa_{\rho_\lb}$, $\rho_\rb\pa_{\rho_\rb}$, and $\pa_{\hat Y}$ obeys the same pointwise bound in view of~\eqref{EqCNConstrCon}. This proves $p\in\cA^{(\alpha_0,\alpha_1+n-1)}(\ff)$.
  
  For the polyhomogeneous version, note that for any $C$, the derivative
  \[
    p_{0,C}(\rho_\lb,\rho_\rb,\hat Y) := \Biggl(\prod_{\genfrac{}{}{0pt}{}{(z,k)\in\cE_0}{\Re z\leq C}}(\rho_\lb\pa_{\rho_\lb}-z)\Biggr)p_0(\rho_\lb,\rho_\rb,\hat Y)
  \]
  obeys the bound~\eqref{EqCNConstrp0Bd} with $\alpha_0$ replaced by $C$, and so do all derivatives of $p_{0,C}$ along any power of $\rho_\lb\pa_{\rho_\lb}$, $\rho_\rb\pa_{\rho_\rb}$, $\pa_{\hat Y}$; indeed, this follows from~\eqref{EqCNConstrPhg}. This, together with an analogous argument at $\rb$, proves $p\in\cA_\phg^{(\cE_0,\cE_1+n-1)}(\ff)$ and thus finishes the proof.
\end{proof}

\section{Elliptic parametrix construction}
\label{SP}

In this section, we shall prove Theorem~\ref{ThmI} as well as a weaker result (Theorem~\ref{ThmB}) when the boundary spectrum is not constant. Thus, $P\in\Psi_0^m(X,\Omegazero^{\frac12}X)$ is fully elliptic at the weight $\alpha\in\R$ in the sense of Definition~\ref{DefIInv}.

\subsection{Constant boundary spectrum}
\label{SsPC}

In this section, we assume that the boundary spectrum $\Specb(P,y)$ of $P$ is independent of the boundary point $y$; we denote it by $\Specb(P)\subset\C\times\N_0$.

At first, we focus on the construction of a right parametrix. The usual symbolic parametrix construction produces $Q_0\in\Psi_0^{-m}(X,\Omegazero^{\frac12}X)$ with the property that
\[
  P Q_0 = I - R_0,\qquad R_0\in\Psi_0^{-\infty}(X,\Omegazero^{\frac12}X).
\]
Passing to reduced normal operators in local coordinates $x\geq 0$, $y\in\R^{n-1}$, $\hat\eta\in\Sph^{n-2}$ on the 0-cosphere bundle, this implies
\[
  \hat N(P,y,\hat\eta)\hat N(Q_0,y,\hat\eta) = I - \hat N(R_0,y,\hat\eta).
\]
We wish to find a residual operator $Q_1$ in the large 0-calculus so that (recalling Proposition~\ref{PropCNSmall})
\begin{equation}
\label{EqPQ1}
  \hat N(P,y,\hat\eta) \hat N(Q_1,y,\hat\eta) = \hat N(R_0,y,\hat\eta) \in \Psi_{\bop,\scop}^{-\infty,(\emptyset,\N_0,\emptyset)}([0,\infty],\Omegab^{\frac12}).
\end{equation}

For the analysis of the reduced normal operator, we need to work with weighted b-scattering Sobolev spaces. Thus, let
\begin{equation}
\label{EqPHbsc}
  H_{\bop,\scop}^{0,(\alpha,r)}([0,\infty],\Omegab^{\frac12}) := \Bigl(\frac{t}{t+1}\Bigr)^\alpha(1+t)^{-r}L^2([0,\infty],\Omegab^{\frac12}[0,\infty]),
\end{equation}
and define $H_{\bop,\scop}^{s,(\alpha,r)}([0,\infty],\Omegab^{\frac12})$ for $s\geq 0$ to consist of those elements of the space~\eqref{EqPHbsc} which remain in this space upon application of any elliptic $s$-th order b-scattering ps.d.o.\ in $\Psi_{\bop,\scop}^{s,(0,0)}([0,\infty],\Omegab^{\frac12})$; for $s<0$, define $H_{\bop,\scop}^{s,(\alpha,r)}([0,\infty],\Omegab^{\frac12})$ as the space of distributions of the form $u_1+A u_2$ where $u_1,u_2\in H_{\bop,\scop}^{0,(\alpha,r)}([0,\infty],\Omegab^{\frac12})$ and $A\in\Psi_{\bop,\scop}^{|s|,(0,0)}([0,\infty],\Omegab^{\frac12})$ is elliptic.

We shall use the notation of index sets $\cE_\pm$, $\wh{\cE_\pm}(j)$, $\wh{\cE_\pm}$, $\wh{\cE_\pm^\flat}$, $\wh{\cE_\pm^\sharp}$, $\wh{\cE_\ff^\pm}$ from Theorem~\ref{ThmI}.

\begin{prop}[Inverse of the normal operator]
\label{PropPNInv}
  The map
  \begin{equation}
  \label{EqPNInv}
    \hat N(P,y,\hat\eta) \colon H_{\bop,\scop}^{s,(\alpha,r)}([0,\infty],\Omegab^{\frac12}) \to H_{\bop,\scop}^{s-m,(\alpha,r-m)}([0,\infty],\Omegab^{\frac12})
  \end{equation}
  is invertible for any $s,r\in\R$. There exists
  \[
    P^-\in\Psi_0^{-m}(X,\Omegazero^{\frac12}X)+\Psi_0^{-\infty,(\wh{\cE_+},\N_0,\wh{\cE_-}+(n-1))}(X,\Omegazero^{\frac12}X)
  \]
  so that $\hat N(P^-,y,\hat\eta) = \hat N(P,y,\hat\eta)^{-1}$ (the inverse of~\eqref{EqPNInv}) for all $y,\hat\eta$.
\end{prop}
\begin{proof}
  We shall construct the difference $P^--Q_0$ by solving~\eqref{EqPQ1}; the choice of the weight $\alpha$ will inform the asymptotics at $\lb$ and $\rb$. Recall that by Proposition~\ref{PropCNSmall}, $\hat N(P,y,\hat\eta)\in\Psi_{\bop,\scop}^{m,(0,m)}([0,\infty],\Omegab^{\frac12})$ is an elliptic operator with $\hat\eta$-independent indicial operator $I(P,y)$.

  \pfstep{(1) Inverting the b-normal operator.} Passing to indicial operators in~\eqref{EqPQ1}, we first solve
  \begin{equation}
  \label{EqPInd}
    I(P,y)I(Q_{1 1},y) = I(R_0,y).
  \end{equation}
  Passing to indicial families and using the full ellipticity at the weight $\alpha$ as well as the lower bound~\eqref{EqCNbsc2Ind}, we may thus set
  \[
    q_{1 1}(y,s) = (2\pi)^{-1} \int_{\Im\sigma=-\alpha} s^{i\sigma} I(P,y,\sigma)^{-1}I(R_0,y,\sigma)\,\dd\sigma.
  \]
  Shifting the contour of integration and using the residue theorem, one finds that $q_{1 1}\in\CI(\R^{n-1}_y,\cA_\phg^{(\cE_+,\cE_-)}([0,\infty]_s))$ where $\cE_+,\cE_-\subset\C\times\N_0$ are defined around~\eqref{EqIEpm} (see also Corollary~\ref{CorCNInd}). Using $\tau=\frac{s-1}{s+1}$ and Corollary~\ref{CorCN0fromInd}, we conclude that there exists
  \begin{equation}
  \label{EqPIndQ11}
    Q_{1 1} \in \Psi_{0'}^{-\infty,(\cE_+,\N_0,\emptyset,\cE_-+(n-1))}(X,\Omegazero^{\frac12}X),\qquad I(Q_{1 1})=q_{1 1}.
  \end{equation}
  Proposition~\ref{PropCC0x} implies that
  \[
    R_{1 1} := I - P(Q_0+Q_{1 1}) = R_0 - P Q_{1 1} \in \Psi_{0'}^{-\infty,(\cE_+,\N_0,\emptyset,\cE_-+(n-1))}(X,\Omegazero^{\frac12}X).
  \]
  Then Proposition~\ref{PropCNbsc} and equation~\eqref{EqPInd} give
  \[
    \hat N(R_{1 1},y,\hat\eta) = \hat N(R_0,y,\hat\eta) - \hat N(P,y,\hat\eta)\hat N(Q_{1 1},y,\hat\eta) \in\Psi_{\bop,\scop}^{-\infty,(\cE_+,\N_0+1,\cE_-)}([0,\infty],\Omegab^{\frac12}).
  \]
  It is important to keep track of the stronger information provided by Lemma~\ref{LemmaCNResidual}, namely that $\wh{r_{1 1}}(y,\tau,\eta)$ (defined via a partial Fourier transform of the de-densitized Schwartz kernel of $R_{1 1}$ restricted to $\ff'_y$, or directly in terms of $\hat N(R_{1 1},y,\hat\eta)$ by the formula~\eqref{EqCNRecover}) is an element of $\sS(\R^{n-1}_\eta;\cA_\phg^{(\cE_+,\cE_-+(n-1))}([-1,1]_\tau))$ and moreover, by construction, vanishes at $\eta=0$.

  \pfstep{(2) Solving away the error at $\lb_\bop$.} Define the index sets
  \begin{equation}
  \label{EqPTildeSets}
    \wt{\cE_\pm}(0) := \cE_\pm\extcup\,\cE_\pm,\qquad
    \wt{\cE_\pm}(j+1) := \cE_\pm\extcup\,\bigl(\wt{\cE_\pm}(j)+1\bigr),\qquad
    \wt{\cE_\pm} := \bigcup_{j=0}^\infty \wt{\cE_\pm}(j).
  \end{equation}
  We claim that there exists
  \begin{equation}
  \label{EqPQ12}
    Q_{1 2}\in\Psi_{0'}^{-\infty,(\wt{\cE_+},\N_0,\emptyset,\emptyset)}(X,\Omegazero^{\frac12}X) = \Psi_0^{-\infty,(\wt{\cE_+},\N_0,\emptyset)}(X,\Omegazero^{\frac12}X)
  \end{equation}
  so that the reduced normal operator of the remaining error
  \begin{equation}
  \label{EqPR2}
    R_{1 2} := I-P(Q_0+Q_{1 1}+Q_{1 2}) = R_{1 1} - P Q_{1 2} \in \Psi_{0'}^{-\infty,(\wt{\cE_+},\N_0,\emptyset,\cE_-+(n-1))}(X,\Omegazero^{\frac12}X)
  \end{equation}
  has trivial index set at $\lb_\bop$, that is,
  \begin{equation}
  \label{EqPR2Norm}
    \hat N(R_{1 2},y,\hat\eta) = \hat N(R_{1 1},y,\hat\eta) - \hat N(P,y,\hat\eta)\hat N(Q_{1 2},y,\hat\eta) \in \Psi_{\bop,\scop}^{-\infty,(\emptyset,\N_0+1,\cE_-)}([0,\infty],\Omegab^{\frac12}).
  \end{equation}
  
  In order to accomplish this, we first make a general observation. If $\cF_{\lb_\bop},\cF_{\ff_{\bop,0}}\subset\C\times\N_0$ are any index sets and $B\in\Psi_{\bop,\scop}^{-\infty,(\cF_{\lb_\bop},\cF_{\ff_{\bop,0}},\emptyset)}([0,\infty],\Omegab^{\frac12})$, then $\hat N(P,y,\hat\eta)\circ B$ is to leading order at $\lb_\bop$ given by the action of $I(P,y)$, lifted to the left factor, on the polyhomogeneous expansion of $B$ at $\lb_\bop$. More precisely, consider local coordinates
  \begin{equation}
  \label{EqPlbCoord}
    s = \frac{t}{t'} \in [0,1),\qquad
    t'\in[0,\infty]
  \end{equation}
  near $\lb_\bop\subset[0,\infty]_{\bop,\scop}^2$. Then $|\frac{\dd t}{t}\frac{\dd t'}{t'}|^{\frac12}=|\frac{\dd s}{s}\frac{\dd t'}{t'}|^{\frac12}$. We define $I_{\lb_\bop}(P,y)$ to have Schwartz kernel given by the same formula as $I(P,y)$ in~\eqref{EqCNbNormal} but in the $s$-coordinates, so
  \[
    I_{\lb_\bop}(P,y) = \wh{K_P^0}\Bigl(\frac{s}{s'},0,0,y\Bigr)\,\Bigl|\frac{\dd s}{s}\frac{\dd s'}{s'}\Bigr|^{\frac12}.
  \]
  This acts fiberwise (on each fiber of $[0,1)\times[0,\infty]\to[0,\infty]$) on extendible distributions with compact support on $[0,1)_s\times[0,\infty]_{t'}$. With $\chi\in\CIc([0,1)\times[0,\infty])\subset\CI([0,\infty]_{\bop,\scop}^2)$ denoting a cutoff, identically $1$ near $\{0\}\times[0,\infty]$, we then have
  \begin{equation}
  \label{EqPlbNormOp}
    \chi\bigl( \hat N(P,y,\hat\eta)\circ B - I_{\lb_\bop}(P,y)\circ(\chi B)\bigr) \in \Psi_{\bop,\scop}^{-\infty,(\cF_{\lb_\bop}+1,\cF_{\ff_{\bop,0}},\emptyset)}([0,\infty],\Omegab^{\frac12}).
  \end{equation}
  This can easily be proved by direct inspection of the integral kernel of $\hat N(P,y,\hat\eta)\circ B$.
  
  Returning to the task at hand, the first step is to find an operator family $B_0(y,\hat\eta)\in\Psi_{\bop,\scop}^{-\infty,(\wh\cE_+(0),\N_0+1,\emptyset)}([0,\infty],\Omegab^{\frac12})$, with Schwartz kernel $B_0(y,\hat\eta)(s,t')|\frac{\dd s}{s}\frac{\dd t'}{t'}|^{\frac12}$ vanishing for $s\geq\half$ (using the coordinates~\eqref{EqPlbCoord}), so that
  \begin{equation}
  \label{EqPlb1}
    \hat N(R_{1 1},y,\hat\eta) - \hat N(P,y,\hat\eta)\circ B_0(y,\hat\eta) \in \Psi_{\bop,\scop}^{-\infty,(\wt{\cE_+}(0)+1,\N_0+1,\cE_-)}([0,\infty],\Omegab^{\frac12})
  \end{equation}
  has index set at $\lb_\bop$ improved by $1$. In view of~\eqref{EqPlbNormOp}, this holds provided
  \[
    \chi\bigl(\hat N(R_{1 1},y,\hat\eta)(s t',t') - I_{\lb_\bop}(P,y)\circ B_0(y,\hat\eta)(s,t')\bigr)  \in \Psi_{\bop,\scop}^{-\infty,(\wt{\cE_+}(0)+1,\N_0+1,\emptyset)}([0,\infty],\Omegab^{\frac12}).
  \]
  (Here the arguments of $\hat N$ are $t=s t'$ and $t'$.) We can explicitly construct such an operator $B_0$ using the (inverse) Mellin transform: with $\chi_0\in\CIc([0,\half))$ identically $1$ on $[0,\tfrac14]$, we define
  \[
    n(R_{1 1},y,\hat\eta)(\sigma,t') := \int_0^\infty s^{-i\sigma} \chi_0(s)\biggl(\hat N(R_{1 1},y,\hat\eta)(s t',t')\Bigl|\frac{\dd s}{s}\frac{\dd t'}{t'}\Bigr|^{-\frac12}\biggr)\,\frac{\dd s}{s}
  \]
  and then set
  \[
    B_0(y,\hat\eta)(s,t') := (2\pi)^{-1}\int_{\Im\sigma=-\alpha} s^{i\sigma} \chi_0(s)I(P,y,\sigma)^{-1} n(R_{1 1},y,\hat\eta)(\sigma,t')\,\dd\sigma\,\Bigl|\frac{\dd s}{s}\frac{\dd t'}{t'}\Bigr|^{\frac12}.
  \]
  Note then that $n(R_{1 1},y,\hat\eta)(\sigma,t')$ is meromorphic in $\sigma$ with divisor $\cE_+$, and thus the integrand defining $B_0(y,\hat\eta)$ is meromorphic and its divisor in $\Im\sigma>-\alpha$ is contained in $\cE_+\extcup\,\cE_+=\wt{\cE_+}(0)$. For later use, we note that if we define
  \begin{equation}
  \label{EqPalpha0}
    \alpha_0:=\min_{(z,k)\in\cE_+}\Re z
  \end{equation}
  and let $\eps\in(0,1)$ be such that $\Re z\notin(\alpha_0-\eps,\alpha_0)$ for all $(z,k)\in\cE_0$, then we may equivalently integrate over $\Im\sigma=-(\alpha_0-\frac{\eps}{2})$.
  
  Having thus arranged~\eqref{EqPlb1}, we next show that the family $B_0(y,\hat\eta)$ is the reduced normal operator of an element of the large extended 0-calculus. With the partially Fourier transformed Schwartz kernels $\hat p$ and $\wh{r_{1 1}}$ of the normal operators of $P,R_{1 1}$ at hand, we need to study (cf.\ \eqref{EqCNRecover} relative to $B_0$, and note that $\frac{t}{t'}=\frac{1+\tau}{1-\tau}$ when $\tau=\frac{t-t'}{t+t'}$)
  \begin{align*}
    \wh{q_{1 2}^{(0)}}(y,\tau,\eta) &:= \Bigl(\frac{1-\tau}{2}\Bigr)^{n-1} \biggl( \Bigl|\frac{\dd t}{t}\frac{\dd t'}{t'}\Bigr|^{-\frac12} B_0\Bigl(y,\frac{\eta}{|\eta|}\Bigr)\biggr) \Bigl(\frac{1+\tau}{1-\tau},\half|\eta|(1-\tau)\Bigr) \\
      &= \Bigl(\frac{1-\tau}{2}\Bigr)^{n-1} (2\pi)^{-1} \int_{\Im\sigma=-\alpha} \Bigl(\frac{1+\tau}{1-\tau}\Bigr)^{i\sigma}\chi_0\Big(\frac{1+\tau}{1-\tau}\Bigr) \\
      &\hspace{11em} \times I(P,y,\sigma)^{-1} n\Bigl(R_{1 1},y,\frac{\eta}{|\eta|}\Bigr)\bigl(\sigma,\half|\eta|(1-\tau)\bigr)\,\dd\sigma \\
      &= \Bigl(\frac{1-\tau}{2}\Bigr)^{n-1}\chi_0\Big(\frac{1+\tau}{1-\tau}\Bigr) (2\pi)^{-1} \\
      &\qquad \times \int_{\Im\sigma=-\alpha} \Bigl(\frac{1+\tau}{1-\tau}\Bigr)^{i\sigma} I(P,y,\sigma)^{-1} \\
      &\qquad\qquad \times \int_0^\infty s'{}^{-i\sigma}\chi_0(s') (1+s')^{n-1} \wh{r_{1 1}}\Bigl(y,\frac{s'-1}{s'+1},(s'+1)\frac{1-\tau}{2}\eta\Bigr)\,\frac{\dd s'}{s'}\,\dd\sigma.
  \end{align*}
  Changing variables via $\kappa=\frac{s'-1}{s'+1}$ (i.e.\ $s'=\frac{1+\kappa}{1-\kappa}$), this equals
  \begin{align*}
      &\Bigl(\frac{1-\tau}{2}\Bigr)^{n-1}\chi_0\Big(\frac{1+\tau}{1-\tau}\Bigr) (2\pi)^{-1} \\
      &\qquad \times \int_{\Im\sigma=-\alpha} \Bigl(\frac{1+\tau}{1-\tau}\Bigr)^{i\sigma} I(P,y,\sigma)^{-1} \int_0^\infty \Bigl(\frac{1+\kappa}{1-\kappa}\Bigr)^{-i\sigma} \tilde r_{1 1}\bigl(y,\kappa,(1-\tau)\eta\bigr)\,\frac{2\,\dd\kappa}{1-\kappa^2}\,\dd\sigma,
  \end{align*}
  where $\tilde r_{1 1}(y,\kappa,\zeta)=\chi_0(\frac{1+\kappa}{1-\kappa})(\frac{1-\kappa}{2})^{-(n-1)}\wh{r_{1 1}}(y,\kappa,\frac{\zeta}{1-\kappa})$ is supported in $\kappa<0$ and thus lies in
  \[
    \sS\bigl(\R_\zeta;\cA_\phg^{(\cE_+,\emptyset)}([-1,1]_\kappa)\bigr)
  \]
  with smooth dependence on $y\in\R^{n-1}$. We thus conclude that
  \[
    \wh{q_{1 2}^{(0)}}(y,\tau,\eta) \in \sS\bigl(\R^{n-1}_\eta;\cA_\phg^{(\wt{\cE_+}(0),\emptyset)}([-1,1]_\tau)\bigr)
  \]
  inherits the smooth dependence on $\eta$ across $\eta=0$ from $\wh{r_{1 1}}(y,\tau,\eta)$, and like the latter vanishes at $\eta=0$. By (the proof of) the second part of Lemma~\ref{LemmaCNResidual}, this implies the existence of an element of
  \[
    Q_{1 2}^{(0)}\in\Psi_{0'}^{-\infty,(\wt{\cE_+}(0),\N_0,\emptyset,\emptyset)}(X,\Omegazero^{\frac12}X),\qquad
    \hat N(Q_{1 2}^{(0)},y,\hat\eta)=B_0(y,\hat\eta).
  \]
  Thus, setting
  \[
    R_{1 1}^{(0)} := R_{1 1} - P Q_{1 2}^{(0)} \in \Psi_{0'}^{-\infty,(\wt{\cE_+}(0),\N_0,\emptyset,\cE_-)}(X,\Omegazero^{\frac12}X),
  \]
  we have $\hat N(R_{1 1}^{(0)},y,\hat\eta)\in\Psi_{\bop,\scop}^{-\infty,(\wt{\cE_+}(0)+1,\N_0+1,\cE_-)}([0,\infty],\Omegab^{\frac12})$, and indeed the sharper regularity across $\eta=0$ captured by Lemma~\ref{LemmaCNResidual}.
  
  We then solve away the improved leading order term of $R_{1 1}^{(0)}$ at $\lb_\bop$ using the same argument, but now taking the Mellin transform along a contour $\Im\sigma=-\alpha'$ lower in the complex plane. Concretely, with $\alpha_0$ defined by~\eqref{EqPalpha0}, we have $\alpha_0+1=\min_{(z,k)\in\wt{\cE_+}(0)+1}\Re z$; we then pick $\eps\in(0,1)$ such that no $(z,k)\in\wt{\cE_+}(0)+1$ has real part in $(\alpha_0+1-\eps,\alpha_0+1)$, and take $\alpha'=\alpha_0+1-\frac{\eps}{2}$. Setting
  \[
    \cE'(1):=(\wt{\cE_+}(0)+1)\extcup\,\{(z,k)\in\cE\colon\Re z\geq\alpha_0+1\}\subset\wt{\cE_+}(1),
  \]
  this produces an operator $B_1(y,\hat\eta)\in\Psi_{\bop,\scop}^{-\infty,(\cE'(1),\N_0+1,\emptyset)}([0,\infty],\Omegab^{\frac12})$ in the range of the reduced normal operator map, and indeed with $\wh{q_{1 2}^{(1)}}$ defined analogously to~\eqref{EqCNRecover} (relative to $B_1$) satisfying
  \[
    \wh{q_{1 2}^{(1)}}(y,\tau,\eta) \in \sS\bigl(\R^{n-1}_\eta;\cA_\phg^{(\cE'(1),\emptyset)}([-1,1]_\tau)\bigr),\qquad \wh{q_{1 2}^{(1)}}(y,\tau,0)=0,
  \]
  with the property that
  \[
    \hat N(R_{1 1}^{(1)},y,\hat\eta) := \hat N(R_{1 1}^{(0)},y,\hat\eta) - \hat N(P,y,\hat\eta)B_1(y,\hat\eta) \in \Psi_{\bop,\scop}^{-\infty,(\cE'(1)+1,\N_0+1,\cE_-)}([0,\infty],\Omegab^{\frac12}).
  \]
  Note that $\cE'(1)+1\subset\{(z,k)\in\wt{\cE_+}(1)+1\colon\Re z\geq\alpha_0+2\}$ encodes one more order of decay than the index set $\wt{\cE_+}(0)+1$ of the previous error term $R_{1 1}^{(0)}$.
  
  Proceeding iteratively, we obtain a sequence of partially Fourier transformed, de-den\-si\-tized Schwartz kernels
  \begin{align*}
    &\wh{q_{1 2}^{(j)}}(y,\tau,\eta) \in \sS\bigl(\R^{n-1}_\eta;\cA_\phg^{(\cE'(j),\emptyset)}([-1,1]_\tau)\bigr),\qquad \wh{q_{1 2}^{(j)}}(y,\tau,0)=0, \\
    &\qquad \cE'(j) := (\cE'(j-1)+1) \extcup\,\{ (z,k)\in\cE \colon \Re z\geq\alpha_0+j \} \subset \wt{\cE_+}(j),
  \end{align*}
  corresponding to reduced normal operators $B_j(y,\hat\eta)\in\Psi_{\bop,\scop}^{-\infty,(\cE'(j),\N_0+1,\emptyset)}([0,\infty],\Omegab^{\frac12})$ via the relationship~\eqref{EqCNRecover}, so that
  \[
    \hat N(R_{1 1}^{(j)},y,\hat\eta) := \hat N(R_{1 1}^{(j-1)},y,\hat\eta) - \hat N(P,y,\hat\eta)B_j(y,\hat\eta) \in \Psi_{\bop,\scop}^{-\infty,(\cE'(j)+1,\N_0+1,\cE_-)}([0,\infty],\Omegab^{\frac12}).
  \]
  
  We then asymptotically sum the $q_{1 2}^{(j)}$, $j=0,1,2,\ldots$, at $\tau=-1$; noting that $\sS(\R^{n-1})=\cA_\phg^\emptyset(\ol{\R^{n-1}})$, this is a standard asymptotic sum for polyhomogeneous conormal distributions on $[-1,0)_\tau\times\ol{\R^{n-1}}$. Thus, there exists $q_{1 2}=q_{1 2}(y,\tau,\eta)$,
  \[
    q_{1 2} \sim \sum_{j\geq 0} q_{1 2}^{(j)} \in \sS\bigl(\R^{n-1}_\eta;\cA_\phg^{(\wt{\cE_+},\emptyset)}([-1,1]_\tau)\bigr),\qquad q_{1 2}(y,\tau,0)=0,
  \]
  with smooth $y$-dependence, corresponding to a reduced normal operator
  \[
    B(y,\hat\eta):=\hat N(Q_{1 2},y,\hat\eta)\in\Psi_{\bop,\scop}^{-\infty,(\wt{\cE_+},\N_0+1,\emptyset)}([0,\infty],\Omegab^{\frac12})
  \]
  of an operator $Q_{1 2}$ as in~\eqref{EqPQ12}, so that~\eqref{EqPR2}--\eqref{EqPR2Norm} hold. While~\eqref{EqPR2} states that the restriction of the Schwartz kernel of $R_{1 2}$ to $\ff'$ has index set $\wt{\cE_+}$ at $\lb'\cap\ff'$, Lemma~\ref{LemmaCNResidual} and the membership~\eqref{EqPR2Norm} show that this can be improved to the index set $\emptyset$. Thus, we have
  \begin{align*}
    &\hat N\bigl(P(Q_0+Q_{1 1}+Q_{1 2}),y,\hat\eta\bigr) = I - \hat N(R_3,y,\hat\eta), \\
    &\qquad R_3 \in \Psi_{0'}^{-\infty,(\emptyset,\N_0,\emptyset,\cE_-+(n-1))}(X,\Omegazero^{\frac12}X) = \Psi_0^{-\infty,(\emptyset,\N_0,\cE_-+(n-1))}(X,\Omegazero^{\frac12}X), \\
    &\qquad \hat N(R_3,y,\hat\eta)=\hat N(R_{1 2},y,\hat\eta) \in \Psi_{\bop,\scop}^{-\infty,(\emptyset,\N_0+1,\cE_-)}([0,\infty],\Omegab^{\frac12}),
  \end{align*}
  
  \pfstep{(3) Solving away the error at $\ff_{\bop,0}$.} We now solve away the error $R_3$ using an asymptotic Neumann series argument. To this end, note that by Proposition~\ref{PropCNbsc}, we have
  \begin{gather}
    R_3^j \in \Psi_0^{-\infty,(\emptyset,\N_0,\extcup^j(\cE_-+(n-1)))}(X,\Omegazero^{\frac12}X), \nonumber\\
  \label{EqPR3j}
    \hat N(R_3^j,y,\hat\eta)=\hat N(R_3,y,\hat\eta)^j \in \Psi_{\bop,\scop}^{-\infty,(\emptyset,\N_0+j,\wh{\cE_-}(j)+(n-1))}([0,\infty],\Omegab^{\frac12})
  \end{gather}
  for any $j\in\N$, where $\extcup^j\cE=\cE\extcup\cdots\extcup\cE$ ($j$ extended unions). The stated membership of $R_3^j$ shows that Lemma~\ref{LemmaCNResidual} is applicable, and in view of the stated membership of $\hat N(R_3^j,y,\hat\eta)$, we may replace $R_3^j$ by an operator
  \[
    R_3^{(j)} \in \Psi_0^{-\infty,(\emptyset,\N_0,\wh{\cE_-}+(n-1))}(X,\Omegazero^{\frac12}X)
  \]
  with $\hat N(R_3^j,y,\hat\eta)=N(R_3^{(j)},y,\hat\eta)$; the point is that the index set of $R_3^{(j)}$ at $\rb$ is fixed.
  
  We wish to take an asymptotic sum of $\hat N(R_3^{(j)},y,\hat\eta)$ at $\ff_{\bop,0}$ while remaining in the range of the reduced normal operator map; this is most easily done by working with the partially Fourier transformed restricted Schwartz kernels $R_3^{(j)}|_{\ff'}$,
  \[
    \wh{r_3^{(j)}}(y,\tau,\eta) \in \sS\bigl(\R^{n-1}_\eta;\cA_\phg^{(\emptyset,\wh{\cE_-}+(n-1))}([-1,1]_\tau)\bigr).
  \]
  Note that by~\eqref{EqCNRecover} and~\eqref{EqPR3j}, we have $|\wh{r_3^{(j)}}(y,\tau,\eta)|\lesssim(1+\tau)^N(1-\tau)^{\beta_0+(n-1)}|\eta|^j$ for $|\eta|<1$, and with $\beta_0<\Re\wh{\cE_-}$ and $N\in\R$; but since $\wh{r_3^{(j)}}$ is smooth in $\eta$, this implies that $\wh{r_3^{(j)}}$ in fact vanishes together with all its derivatives of order $\leq j-1$ at $\eta=0$. We record this as
  \[
    \wh{r_3^{(j)}} \in \cI^j\sS\bigl(\R^{n-1}_\eta;\cA_\phg^{(\emptyset,\wh{\cE_-}+(n-1))}([-1,1]_\tau)\bigr),
  \]
  where $\cI^j\sS(\R^{n-1})$ is the space of Schwartz functions vanishing at $0$ together with all their derivatives of order $\leq j-1$.
  
  We can thus asymptotically sum the $\wh{r_3^{(j)}}$ at $\eta=0$, obtaining
  \begin{align*}
    &\tilde r_3(y,\tau,\eta) \in \sS\bigl(\R_\eta^{n-1};\cA_\phg^{(\emptyset,\wh{\cE_-}+(n-1))}([-1,1]_\tau)\bigr), \\
    &\tilde r_3 - \sum_{j=1}^{N-1} \wh{r_3^{(j)}} \in \cI^N\sS\bigl(\R^{n-1};\cA_\phg^{(\emptyset,\wh{\cE_-}+(n-1))}([-1,1])\bigr),\qquad N\in\N.
  \end{align*}
  Upon taking the inverse Fourier transform of $\tilde r_3$ in $\eta$, we obtain, upon extension off $\ff'$, an operator
  \begin{align}
  \label{EqPIndR3}
    &\tilde R_3 \in \Psi_0^{-\infty,(\emptyset,\N_0,\wh{\cE_-}+(n-1))}(X,\Omegazero^{\frac12}X), \\
    &\hat N(\tilde R_3,y,\hat\eta) - \sum_{j=1}^{N-1} \hat N(R_3^j,y,\hat\eta) \in \Psi_{\bop,\scop}^{-\infty,(\emptyset,\N_0+N,\wh{\cE_-})}([0,\infty],\Omegab^{\frac12}),\qquad N\in\N. \nonumber
  \end{align}
  Therefore, $I+\tilde R_3$ is an approximate inverse of $I-R_3$ on the reduced normal operator level, in the sense that
  \[
    (I-R_3)(I+\tilde R_3) = I - R'
  \]
  where
  \begin{equation}
  \label{EqPRprime}
  \begin{split}
    &R'\in\Psi_0^{-\infty,(\emptyset,\N_0,\wh{\cE_-}+(n-1))}(X,\Omegazero^{\frac12}X), \\
    &\hat N(R',y,\hat\eta)\in\Psi_{\bop,\scop}^{-\infty,(\emptyset,\emptyset,\wh{\cE_-})}([0,\infty],\Omegab^{\frac12}) = \Psi^{-\infty,(\emptyset,\wh{\cE_-})}([0,\infty],\Omegab^{\frac12}).
  \end{split}
  \end{equation}

  \pfstep{(4) Full right parametrix.} The memberships~\eqref{EqPIndQ11} and \eqref{EqPQ12} give $Q_0+Q_{1 1}+Q_{1 2}\in\Psi_0^{-m}(X,\Omegazero^{\frac12}X)+\Psi_0^{-\infty,(\wt{\cE_+},\N_0,\cE_-+(n-1))}(X,\Omegazero^{\frac12}X)$. Enlarging the index set of this operator as well as of~\eqref{EqPIndR3} at $\rb$ to $\wt{\cE_-}+(n-1)$ for symmetry reasons (already having in mind the construction of a left parametrix below), Proposition~\ref{PropCC0} shows that
  \[
    (Q_0+Q_{1 1}+Q_{1 2})(I+\tilde R_3) \in \Psi_0^{-m}(X,\Omegazero^{\frac12}X) + \Psi_0^{-\infty,\tilde\cE''}(X,\Omegazero^{\frac12}X),
  \]
  with $\tilde\cE''=(\tilde\cE_\lb,\tilde\cE'_\ff,\tilde\cE'_\rb)$ where we can take
  \[
    \tilde\cE_\lb=\wt{\cE_+}\extcup\wt{\cE_+},\quad
    \tilde\cE'_\ff=\N_0\extcup\,\bigl(\wt{\cE_+}+\wt{\cE_-}+(n-1)\bigr),\quad
    \tilde\cE'_\rb=(\wt{\cE_-}\extcup\wt{\cE_-}) + (n-1).
  \]
  Setting $\tilde\cE'=(\tilde\cE_\lb,\N_0,\tilde\cE'_\rb)$, extension off $\ff$ gives an operator
  \begin{equation}
  \label{EqPQ1prime}
    Q_1' \in \Psi_0^{-m}(X,\Omegazero^{\frac12}X) + \Psi_0^{-\infty,\tilde\cE'}(X,\Omegazero^{\frac12}X)
  \end{equation}
  with the property that
  \[
    N(P Q_1',y,\hat\eta) = I - N(R',y,\hat\eta)
  \]
  with $R'$ as in~\eqref{EqPRprime}.

  \pfstep{(5) Left parametrix; true inverse.} Analogous arguments (or application of the right parametrix construction to $P^*$ followed by taking adjoints) produce a left parametrix $\tilde Q_1'$ in the same space as $Q_1'$ in~\eqref{EqPQ1prime} with
  \begin{align*}
    &N(\tilde Q_1' P,y,\hat\eta) = I - N(\tilde R',y,\hat\eta), \\
    &\qquad\tilde R' \in \Psi_0^{-\infty,(\wt{\cE_+},\N_0,\emptyset)}(X,\Omegazero^{\frac12}X), \qquad
    \hat N(\tilde R',y,\hat\eta) \in \Psi^{-\infty,(\wt{\cE_+},\emptyset)}([0,\infty],\Omegab^{\frac12}).
  \end{align*}
  
  Since $\hat N(R',y,\hat\eta)$ maps $H_{\bop,\scop}^{s-m,(\alpha,r-m)}([0,\infty],\Omegab^{\frac12})$ (for any $s,r$) into $\cA_\phg^\emptyset([0,\infty],\Omegab^{\frac12})$, it is a compact operator on $H_{\bop,\scop}^{s-m,(\alpha,r-m)}([0,\infty],\Omegab^{\frac12}X)$; similarly, $\hat N(\tilde R',y,\hat\eta)$ is a compact operator on $H_{\bop,\scop}^{s,(\alpha,r)}([0,\infty],\Omegab^{\frac12})$ (as it maps this space into $H_{\bop,\scop}^{\infty,(\alpha+\eps,\infty)}([0,\infty],\Omegab^{\frac12})$ for $\eps>0$ chosen so small that $\alpha+\eps<\Re\wt{\cE_+}$). Therefore,
  \begin{equation}
  \label{EqPNormOpMap}
    \hat N(P,y,\hat\eta)\colon H_{\bop,\scop}^{s,(\alpha,r)}([0,\infty],\Omegab^{\frac12})\to H_{\bop,\scop}^{s-m,(\alpha,r-m)}([0,\infty],\Omegab^{\frac12})
  \end{equation}
  is Fredholm, with approximate right, resp.\ left inverse $\hat N(Q_1',y,\hat\eta)$, resp.\ $\hat N(\tilde Q_1',y,\hat\eta)$. The full ellipticity assumption on $P$ now implies that $\hat N(P,y,\hat\eta)$ has trivial kernel and cokernel (for the particular value of $\alpha$, but for any $s,r$), and thus is invertible. The Schwartz kernel of the inverse of~\eqref{EqPNormOpMap} can then be related to the left and right parametrices constructed above via
  \begin{equation}
  \label{EqPNormOpMap2}
    \hat N(P,y,\hat\eta)^{-1} = \hat N(Q_1',y,\hat\eta) + \hat N(\tilde Q_1',y,\hat\eta)\hat N(R',y,\hat\eta) + \hat N(\tilde R',y,\hat\eta) \hat N(P,y,\hat\eta)^{-1} \hat N(R',y,\hat\eta).
  \end{equation}

  We first claim that
  \begin{equation}
  \label{EqPIdeal}
    \hat N(\tilde R',y,\hat\eta)\hat N(P,y,\hat\eta)^{-1}\hat N(R',y,\hat\eta) \in \Psi^{-\infty,(\wt{\cE_+},\wt{\cE_-})}([0,\infty],\Omegab^{\frac12}),
  \end{equation}
  with smooth dependence on $y$ and $\hat\eta$. Indeed, the Schwartz kernel of~\eqref{EqPIdeal} can be computed by applying $\hat N(P,y,\hat\eta)^{-1}$ and then the (smoothing) operator $N(\tilde R',y,\hat\eta)$ to the restrictions of the Schwartz kernel of $N(R',y,\hat\eta)$ to level sets of $t'$ (see step (6) below for a more delicate version of this argument); the statement~\eqref{EqPIdeal} follows from this description. Smoothness in $y$ and $\hat\eta$ follows by direct differentiation and repeated application of (the inverse of)~\eqref{EqPNormOpMap}.
  
  By Proposition~\ref{PropCNConstr} then, the operator~\eqref{EqPIdeal} lies in the range of the reduced normal operator; it is equal to $N(R'',y,\hat\eta)$ for some $R''\in\Psi_0^{-\infty,(\wt{\cE_+},\N_0,\wt{\cE_-}+(n-1))}(X,\Omegazero^{\frac12}X)$. Taking $P^-$ to be an extension off $\ff$ of $Q_1'+\tilde Q_1'R'+R''$ (cf.\ \eqref{EqPNormOpMap2}), we have thus shown that
  \[
    \hat N(P,y,\hat\eta)^{-1} = \hat N(P^-,y,\hat\eta),\qquad
    P^- \in \Psi_0^{-m}(X,\Omegazero^{\frac12}X) + \Psi_0^{-\infty,\tilde\cE}(X,\Omegazero^{\frac12}X),
  \]
  where $\tilde\cE=(\tilde\cE_\lb,\N_0,\tilde\cE_\rb)$ with $\tilde\cE_\rb:=\tilde\cE'_\rb\extcup\,(\wt{\cE_-}+(n-1))$ in view of Proposition~\ref{PropCC0}.

  \pfstep{(6) Sharpening of the index sets.} We first claim that the index set of $P^-|_\ff$ at $\lb$ is contained in $\wh{\cE_+}$; to prove this, we need to show that $\hat N(P^-,y,\hat\eta)$ has index set $\wh{\cE_+}$ at $\lb_\bop$ for all $y,\hat\eta$. For this purpose, fix $t'_0\in(0,\infty)$ and a cutoff function $\chi\in\CIc([0,\infty))$ with $\chi(t)=1$ for $t<\half t'_0$ and $\chi(t)=0$ for $t>\frac34 t'_0$. In the coordinates $t\geq 0$, $t'\in(0,\infty)$ on $[0,\infty]_{\bop,\scop}^2$ near $\lb_\bop\setminus(\ff_{\bop,0}\cup\ff_{\bop,\infty})$, the cut-off restriction
  \[
    q_{t'_0}(t) := \chi(t)p^-_{t'_0}(t)\in\cA_\phg^{\tilde\cE_\lb}([0,t'_0)_t,\Omegab^{\frac12}),\qquad p^-_{t'_0}(t):=\Bigl|\frac{\dd t'}{t'}\Bigr|^{-\frac12}\hat N(P^-,y,\hat\eta)(t,t'_0),
  \]
  lies in the approximate nullspace of $\hat N(P,y,\hat\eta)$; to wit,
  \[
    \hat N(P,y,\hat\eta)q_{t_0'} = \chi\hat N(P,y,\hat\eta)p_{t'_0}^- + [\hat N(P,y,\hat\eta),\chi] p^-_{t_0'} \in \cA_\phg^\emptyset\bigl([0,\infty),\Omegab^{\frac12}[0,\infty)\bigr).
  \]
  Here we used that $\hat N(P,y,\hat\eta)p_{t'_0}^-$ is supported at $t=t'_0$ and thus outside of $\supp\chi$; and for the second term we also used that the Schwartz kernel of $[\hat N(P,y,\hat\eta),\chi]$ vanishes to infinite order at all boundary hypersurfaces of $[0,\infty]_{\bop,\scop}$ and is supported in $t,t'<t'_0$. But as in step (2) of the proof, we can then use the (inverse) Mellin transform and the indicial family to show that $q_{t'_0}\in\cA_\phg^{\wh{\cE_+}}([0,\infty),\Omegab^{\frac12}[0,\infty))$. This proves the claim. By extension off $\ff$, we may thus replace the index set $\tilde\cE_\lb$ of $P^-$ at $\lb$ by $\wh{\cE_+}$.

  The index set of $P^-|_\ff$ at $\rb$ can be similarly improved by using the fact that
  \[
    \hat N(P,y,\hat\eta)^*\hat N(P^-,y,\hat\eta)^*=0,
  \]
  where we define adjoints with respect to the $L^2([0,\infty],\Omegab^{\frac12})$ inner product. Note that $(z,k)\in\Specb(\hat N(P,y,\hat\eta)^*)$ if and only if $(-\bar z,k)\in\Specb(\hat N(P,y,\hat\eta))$; indeed, the inverse of
  \[
    I\bigl(\hat N(P,y,\hat\eta)^*,y,\sigma\bigr)=\overline{I\bigl(\hat N(P,y,\hat\eta),y,\bar\sigma\bigr)}
  \]
  has a pole at $-i z$ of order $\geq k+1$ iff $i\bar z$ is a pole of $I(\hat N(P,y,\hat\eta),y,\sigma)^{-1}$ of order $\geq k+1$. We thus conclude that $P^*$ is fully elliptic at the weight $-\alpha$, and moreover that
  \begin{align*}
    &\{ (z,k) \in \Specb(\hat N(P,y,\hat\eta)^*) \colon \Re z>-\alpha \} \\
    &\qquad = \{ (-z,k) \colon (\bar z,k) \in \Specb(\hat N(P,y,\hat\eta)),\ \Re z<\alpha \} = \{(\bar z,k) \colon (z,k)\in\cE_- \}.
  \end{align*}

  On the other hand, the index set of $\hat N(P^-,y,\hat\eta)^*$ at $\lb_\bop$ is $\{(\bar z,k)\colon (z,k)\in\tilde\cE_\rb\}$. Following the previous arguments thus shows that we can shrink the index set of $\hat N(P^-,y,\hat\eta)^*$ at $\lb_\bop$ to the set of $(z,k)$ so that $(\bar z,k)\in\wh{\cE_-}$; that is, we can replace $\tilde\cE_\rb$ by $\wh{\cE_-}$. The proof is complete.
\end{proof}

\begin{proof}[End of proof of Theorem~\usref{ThmI}]
  We may now apply Proposition~\ref{PropPNInv} to the task~\eqref{EqPQ1}: we set $Q_1=P^- R_0\in\Psi_0^{-\infty,(\wh{\cE_+},\N_0,\wh{\cE_-}+(n-1))}(X,\Omegazero^{\frac12}X)$. We then have
  \[
    P(Q_0+Q_1) = I-R_1,\qquad R_1:=R_0-P Q_1 \in \Psi_0^{-\infty,(\wh{\cE_+},\N_0+1,\wh{\cE_-}+(n-1))}(X,\Omegazero^{\frac12}X),
  \]
  since all normal operators of $R_1$ vanish (see also Remark~\ref{RmkCNDecReg}).
  
  \pfstep{Solving away the error at $\lb$.} This is analogous to the corresponding step in the inversion of the normal operator and involves, via the Mellin transform, the inversion of the indicial family $I(P,y,\sigma)$. Thus, there exists $Q_2\in\Psi_0^{-\infty,(\wh{\cE_+},\N_0+1,\emptyset)}(X,\Omegazero^{\frac12}X)$ so that
  \[
    P(Q_0+Q_1+Q_2) = I-R_2,\qquad
    R_2:=R_1-P Q_1 \in \Psi_0^{-\infty,(\emptyset,\N_0+1,\wh{\cE_-}+(n-1))}(X,\Omegazero^{\frac12}X).
  \]

  Here, naive accounting of index sets would suggest merely
  \[
    Q_2 \in \Psi_0^{-\infty,(\wt{\cE_+}(0),\N_0+1,\emptyset)}(X,\Omegazero^{\frac12}X),
  \]
  where we recall~\eqref{EqPTildeSets}. The fact that we can take $\wh{\cE_+}$ as the index set of $Q_2$ at $\lb$ can be seen by adapting the argument in step (6) of the proof of Proposition~\ref{PropPNInv} to the 0-setting; the only difference is that now we need to localize also in the boundary variables. Thus, denote the Schwartz kernel of $Q_0+Q_1+Q_2$ by
  \[
    q(x,y,x',y')\Bigl|\frac{\dd x}{x}\frac{\dd y}{x^{n-1}}\frac{\dd x'}{x'}\frac{\dd y'}{x'{}^{n-1}}\Bigr|^{\frac12},
  \]
  fix $z'_0=(x'_0,y'_0)$ with $x'_0>0$ and $y'_0\in\R^{n-1}$, and consider the restriction
  \[
    q_{z'_0} = q(-,-,x'_0,y'_0)\Bigl|\frac{\dd x}{x}\frac{\dd y}{x^{n-1}}\Bigr|^{\frac12} \in \cA_\phg^{\wt{\cE_+}(0)}(X,\Omegazero^{\frac12}X)
  \]
  Let 
  \[
    K = [0,x'_0)_x \times \{ y\in\R^{n-1} \colon |y-y'_0|<x'_0 \}
  \]
  and fix a sequence of cutoff functions $\chi_0,\chi_1,\ldots\in\CIc(K)$ so that
  \[
    \chi_j(x,y)=1,\qquad x\leq \half x'_0,\quad |y-y'_0|<\half x'_0;\qquad
    \supp\chi_j\subset\{\chi_{j-1}=1\}.
  \]
  We have $\chi_0 q_{z'_0}\in\cA_\phg^{\wt{\cE_+}(0)}(K,\Omegazero^{\frac12}K)$. We claim that, for $j\in\N_0$,
  \begin{equation}
  \label{EqPchij}
    \chi_{j+1} P(\chi_j q_{z'_0}) = \chi_{j+1} P q_{z'_0} + \chi_{j+1}[P,\chi_j]q_{z'_0} \in \cA_\phg^\emptyset(K,\Omegazero^{\frac12}K).
  \end{equation}
  Indeed, for the first term, we note that $P q_{z'_0}$ is given by the restriction of the Schwartz kernel of $I-R_2$ to $(x',y')=z'_0$ (ignoring half-density factors) and thus vanishes to infinite order at $\pa X$ and has singular support equal to $(x,y)=z'_0$, thus outside of $\supp\chi_{j+1}$. For the second term, note simply that $\chi_{j+1}[P,\chi_j]\in\Psi_0^{m-1,(\emptyset,\emptyset,\emptyset)}(X,\Omegazero^{\frac12}X)$ and use Lemma~\ref{LemmaCPhg}. Starting with $j=0$ in~\eqref{EqPchij} and passing to indicial operators shows that $\chi_0 q_{z'_0}$ must lie in $\cA_\phg^{\cE_+}$ up to an error term in $\cA^{\alpha_0+1}$ where $\alpha_0$ was defined in~\eqref{EqPalpha0}. Using this information in equation~\eqref{EqPchij} with $j=1$ gives $\chi_1 q_{z'_0}\equiv\cA_\phg^{\wh{\cE_+}(1)}\bmod\cA^{\alpha_0+2}$; and upon using induction, we conclude that $\chi_\infty q_{z'_0}\in\cA_\phg^{\wh{\cE_+}}(K,\Omegazero^{\frac12}K)$ where $\chi_\infty\in\CIc(K)$ has $\supp\chi_\infty\subset\{\chi_j=1\}$ for all $j$. Since we can take $\chi_\infty$ to be equal to $1$ in a neighborhood of $(0,y'_0)$, and since $x'_0$ and $y'_0$ were arbitrary, we conclude that we can reduce the index set of $Q_2$ at $\lb$ to $\wh{\cE_+}$, as claimed.
  
  \pfstep{Solving away the error at $\ff$.} This is again an asymptotic Neumann series argument. Let $\Psi_0^{-\infty,(\emptyset,\N_0+1,\wh{\cE_-^\flat}+(n-1))}(X,\Omegazero^{\frac12}X)\ni\tilde R_2\sim\sum_{j=1}^\infty R_2^j$, where we recall $\wh{\cE_-^\flat}=\wh{\cE_-}\extcup\,(\wh{\cE_-}+1)\extcup\cdots$. Then Proposition~\ref{PropCC0} and the fact that $Q_0+Q_1+Q_2\in\Psi_0^{-m}(X,\Omegazero^{\frac12}X)+\Psi_0^{-\infty,(\wh{\cE_+},\N_0,\wh{\cE_-}+(n-1))}(X,\Omegazero^{\frac12}X)$ give
  \begin{equation}
  \label{EqPQ}
    Q := (Q_0+Q_1+Q_2)(I+\tilde R_2) \in \Psi_0^{-m}(X,\Omegazero^{\frac12}X) + \Psi_0^{-\infty,(\wh{\cE_+},\wh{\cE_\ff^-},\wh{\cE_-^\sharp})}(X,\Omegazero^{\frac12}X);
  \end{equation}
  and we have
  \[
    P Q = I - R,\qquad R\in\Psi_0^{-\infty,(\emptyset,\emptyset,\wh{\cE_-^\flat}+(n-1))}(X,\Omegazero^{\frac12}X) \subset \Psi^{-\infty,(\emptyset,\wh{\cE_-^\flat}+(n-1))}(X,\Omegazero^{\frac12}X).
  \]
  
  The construction of a left parametrix is analogous, or one can take $Q'$ to be the adjoint of a full right parametrix for $P^*$. The proof of Theorem~\ref{ThmI} is complete.
\end{proof}

\begin{rmk}[Index set at the right boundary]
\label{RmkPrb}
  The index set of the right parametrix $Q$ at $\rb$ can in general not be reduced without modifying $Q$. Indeed, note that one can add to $Q$ any element of $\Psi_0^{-\infty,(\emptyset,\cF)}(X,\Omegazero^{\frac12}X)$ for any index set $\cF$, and one will still have $P Q=I-R$ with $R$ fully residual. That is, sharpening $Q$ at $\rb$ to a smaller index set requires either more careful bookkeeping or a suitable modification of $Q$. (An analogous comment applies to the index set of the left parametrix $Q'$ at $\lb$.)
\end{rmk}

Relative to the intrinsic $L^2$-space on 0-$\half$-densities $L^2(X,\Omegazero^{\frac12}X)$, we define weighted 0-Sobolev spaces in the usual manner. Thus, let $\rho\in\CI(X)$ denote a boundary defining function. For $a\in\R$ and $s\geq 0$, we then set
\[
  \rho^a H_0^s(X,\Omegazero^{\frac12}X) = \bigl\{ u=\rho^a u_0 \colon u_0\in L^2(X,\Omegazero^{\frac12}X),\ A u_0\in L^2(X,\Omegazero^{\frac12}X) \bigr\},
\]
where $A\in\Psi_0^s(X,\Omegazero^{\frac12}X)$ has an elliptic principal symbol; and for $s<0$, we define the space $\rho^a H_0^s(X,\Omegazero^{\frac12}X)$ to consist of all sums $u_0+A u_1$ where $u_0,u_1\in\rho^a L^2(X,\Omegazero^{\frac12}X)$, with $A\in\Psi_0^{|s|}(X,\Omegazero^{\frac12}X)$ any (fixed) operator with an elliptic principal symbol. We remark that $L^2(X,\Omegazero^{\frac12}X)=\rho^{\frac{n-1}{2}}L^2(X,\Omegab^{\frac12}X)$ (with equivalent norms), similarly for weighted 0-Sobolev spaces valued in b-$\half$-densities.

\begin{cor}[Generalized inverse]
\label{CorPGen}
  Suppose $X$ is compact. Let $P\in\Psi_0^m(X,\Omegazero^{\frac12}X)$ be fully elliptic at the weight $\alpha$. Then
  \begin{equation}
  \label{EqPGenP}
    P\colon\rho^{\alpha-\frac{n-1}{2}} H_0^s(X,\Omegazero^{\frac12}X)\to\rho^{\alpha-\frac{n-1}{2}} H_0^{s-m}(X,\Omegazero^{\frac12}X)
  \end{equation}
  (i.e.\ $P\colon\rho^\alpha H_0^s(X,\Omegab^{\frac12}X)\to \rho^\alpha H_0^{s-m}(X,\Omegab^{\frac12}X)$) is a Fredholm operator. Denote by
  \[
    G\colon\rho^{\alpha-\frac{n-1}{2}}H_0^{s-m}(X,\Omegazero^{\frac12}X)\to\rho^{\alpha-\frac{n-1}{2}}H_0^s(X,\Omegazero^{\frac12}X)
  \]
  its generalized inverse; that is, $G|_{(\ran P)^\perp}\equiv 0$, and for $f\in\ran P$ we set $G f=u$ where $u$ is the unique solution of $P u=f$ with $u\perp\ker P$; here, orthogonal complements are defined with respect to the $\rho^{\alpha-\frac{n-1}{2}}L^2(X,\Omegazero^{\frac12}X)$ inner product. Then, with index sets defined as in Theorem~\usref{ThmI}, we have
  \[
    G\in\Psi_0^{-m}(X,\Omegazero^{\frac12}X)+\Psi_0^{-\infty,(\wh{\cE_+}\extcup\,(\wh{\cE_-}+2\alpha),\cE_\ff,(\wh{\cE_-}\extcup\,(\wh{\cE_+}-2\alpha)+(n-1))}(X,\Omegazero^{\frac12}X),
  \]
  where with $\wh{\cE_\ff}:=\N_0\extcup\,(\wh{\cE_+^\sharp}+\wh{\cE_-^\sharp}+(n-1))$ we put\footnote{We make no attempt to optimize the index set $\cE_\ff$ here. We merely point out that $\cE_\ff$ is equal to the union of $\N_0$ with an index set all of whose elements $(z,k)$ have $\Re z>n-1$.}
  \[
    \cE_\ff=\wh{\cE_\ff} \cup \Bigl( \bigl(\wh{\cE_\ff}+\bigl[2(\wh{\cE_+^\flat}-\alpha)\cup 2(\wh{\cE_-^\flat}+\alpha)\bigr]+(n-1)\bigr) \extcup\,\bigl(\wh{\cE_+^\sharp}+\wh{\cE_-^\sharp}+(n-1)\bigr)\Bigr).
  \] Furthermore, the orthogonal projections $\Pi=I-G P$ to the nullspace of $P$ in~\eqref{EqPGenP} and $\Pi'=I-P G$ to the orthogonal complement of the range of $P$ satisfy
  \begin{equation}
  \label{EqPGenProj}
  \begin{split}
    \Pi &\in \Psi^{-\infty,(\wh{\cE_+},\wh{\cE_+}-2\alpha+(n-1))}(X,\Omegazero^{\frac12}X), \\
    \Pi' &\in \Psi^{-\infty,(\wh{\cE_-}+2\alpha,\wh{\cE_-}+(n-1))}(X,\Omegazero^{\frac12}X).
  \end{split}
  \end{equation}
  If $P$ is invertible, then $G=P^{-1}\in\Psi_0^{-m}(X,\Omegazero^{\frac12}X)+\Psi_0^{-\infty,(\wh{\cE_+},\wh{\cE_\ff},\wh{\cE_-}+(n-1))}(X,\Omegazero^{\frac12}X)$.
\end{cor}
\begin{proof}
  Let $Q,Q',R,R'$ be as in Theorem~\ref{ThmI}. Thus $P Q=I-R$, with $R$ a compact operator on $\rho^{\alpha-\frac{n-1}{2}}H_0^{s-m}(X,\Omegazero^{\frac12}X)$ (since it maps $\rho^{\alpha-\frac{n-1}{2}}H_0^{s-m}(X,\Omegazero^{\frac12}X)\to\cA_\phg^\emptyset(X,\Omegazero^{\frac12}X)$, and the inclusion of this space back into $\rho^{\alpha-\frac{n-1}{2}}H_0^{s-m}(X,\Omegazero^{\frac12}X)$ is compact); likewise $Q' P=I-R'$, with $R'$ compact on $\rho^{\alpha-\frac{n-1}{2}}H_0^s(X,\Omegazero^{\frac12}X)$ (as it maps this space into $\cA^{\alpha+\eps}(X,\Omegazero^{\frac12}X)$ for some small $\eps>0$, the inclusion of which into $\rho^{\alpha-\frac{n-1}{2}}H_0^s(X,\Omegazero^{\frac12}X)$ is compact). Therefore, $P$ is Fredholm.

  Now, every $u\in\ker P$ in the domain of~\eqref{EqPGenP} satisfies $u=-R' u\in\cA^\alpha(X,\Omegazero^{\frac12}X)$; but then $P u=0$ in fact implies $u\in\cA_\phg^{\wh{\cE_+}}(X,\Omegazero^{\frac12}X)$ in view of the same arguments as around~\eqref{EqPchij}. Thus, if $u_1,\ldots,u_N\in\cA_\phg^{\wh{\cE_+}}(X,\Omegazero^{\frac12}X)$ is an orthonormal basis of $\ker P$, then the orthogonal projection $\Pi=\sum_{j=1}^N \la-,u_j\ra_{\rho^{\alpha-\frac{n-1}{2}}L^2(X,\Omegazero^{\frac12}X)}u_j$ to $\ker P$ has Schwartz kernel
  \[
    X\times X \ni (z,z') \mapsto \sum_{j=1}^N u_j(z) \rho'{}^{-2\alpha+(n-1)}\overline{u_j}(z'),
  \]
  where $\rho'$ is the lift of the boundary defining function of $X$ to the second factor. This implies the membership of $\Pi$ in~\eqref{EqPGenProj}.

  To control $\Pi'$, we consider its adjoint $(\Pi')^*=I-G^*P^*$, which is the orthogonal projection to $\ker P^*$ in $\rho^{-\alpha+\frac{n-1}{2}}L^2(X,\Omegazero^{\frac12}X)$. A calculation based on the relationship $\Omegazero X=\rho^{-(n-1)}\,\Omegab X$ shows that the boundary spectrum of $P^*\in\Psi_0^m(X,\Omegazero^{\frac12}X)$ is given by
  \begin{equation}
  \label{EqPSpecStar}
    \Specb(P^*)=\bigl\{(z,k)\colon(-\bar z+(n-1),k)\in\Specb(P)\bigr\}.
  \end{equation}
  Now, if $v\in\rho^{-\alpha+\frac{n-1}{2}}L^2(X,\Omegazero^{\frac12}X)\cap\ker P^*$, then
  \[
    v=Q^*P^*v+R^*v=R^*v\in\cA^{-\alpha+(n-1)}(X,\Omegazero^{\frac12}X).
  \]
  But by~\eqref{EqPSpecStar}, the smallest index set containing all elements $(z,k)\in\Specb(P^*)$ with $\Re z>-\alpha+(n-1)$ is equal to $\ol{\cE_-}+(n-1)$ where $\ol{\cE_-}=\{(\bar z,k)\colon(z,k)\in\cE_-\}$, and therefore $P^*v=0$ implies $v\in\cA_\phg^{\ol{\wh{\cE_-}}+(n-1)}(X,\Omegazero^{\frac12}X)$. Therefore,
  \[
    (\Pi')^* \in \Psi^{-\infty,(\ol{\wh{\cE_-}}+(n-1),\ol{\wh{\cE_-}}+2\alpha)}(X,\Omegazero^{\frac12}X),
  \]
  which implies~\eqref{EqPGenProj}.

  The generalized inverse $G$ is related to the left and right parametrices via
  \begin{align*}
    &G = (Q'P+R')G = Q'(I-\Pi')+R'G(P Q+R) \\
    &\qquad =Q'(I-\Pi')+R'(I-\Pi)Q+R'G R = Q'+R'Q+R'G R - Q'\Pi'-R'\Pi Q;
  \end{align*}
  note then that
  \[
    R'Q\in\Psi_0^{-\infty,(\wh{\cE_+^\flat},\wh{\cE_+^\flat}+\wh{\cE_-^\sharp}+(n-1),\wh{\cE_-^\sharp}+(n-1))}(X,\Omegazero^{\frac12}X)
  \]
  and also $R'G R\in\Psi^{-\infty,(\wh{\cE_+^\flat},\wh{\cE_-^\flat}+(n-1))}(X,\Omegazero^{\frac12}X)$ lies in this space. If $P$ is invertible and therefore $\Pi=0$, $\Pi'=0$, then from $P G=I$ and $P^*G^*=I$ we conclude that the index set of $G$ at $\lb$, resp.\ $\rb$ can be reduced to $\wh{\cE_+}$, resp.\ $\wh{\cE_-}+(n-1)$.

  If $P$ is not invertible, we compute using Proposition~\ref{PropCC0} that
  \begin{align*}
    Q'\Pi',\ R'\Pi Q \in \Psi^{-\infty,(\cE_\lb,\cE_\ff,\cE_\rb)}(X,\Omegazero^{\frac12}X)
  \end{align*}
  for some (explicit) index sets $\cE_\lb,\cE_\rb$ with $\Re\cE_\lb>\alpha$ and $\Re\cE_\rb>-\alpha+(n-1)$, and with (the somewhat wasteful index set) $\cE_\ff$ given in the statement of the Corollary. The index set of $G$ at $\lb$ can then be improved using $P G=I-\Pi'$ by a simple adaptation of the arguments around~\eqref{EqPchij}: the restriction $G_{z'_0}$ of the Schwartz kernel of $G$ to the preimage of $z'_0\in X^\circ$ under the projection $X\times X^\circ\to X^\circ$ after division by the lift of a positive $\half$-density on $X^\circ$ satisfies $P(\chi_0 G_{z'_0})\in\cA_\phg^{\wh{\cE_-}+2\alpha}(X,\Omegazero^{\frac12}X)$ where $\chi_0\in\CIc(X)$ cuts off to a neighborhood of $z'_0$; therefore $\chi_0 G_{z'_0}\in\cA_\phg^{\wh{\cE_+}\extcup\,(\wh{\cE_-}+2\alpha)}(X,\Omegazero^{\frac12}X)$ by indicial operator arguments. Thus, the index set of $G$ at $\lb$ can be reduced to $\wh{\cE_+}\extcup\,(\wh{\cE_-}+2\alpha)$. The argument at $\rb$ is analogous upon passing to the adjoint equation $P^*G^*=I-(\Pi')^*$.
\end{proof}

\subsection{Parametrices in the calculus with bounds}
\label{SsPB}

In the case that the fully elliptic 0-ps.d.o.\ $P$ does not have constant boundary spectrum, the Schwartz kernel of detailed parametrices cannot be polyhomogeneous anymore (see however \cite{KrainerMendozaBundle}). In this case, one can still construct a parametrix in the calculus with bounds.

\begin{thm}[Precise parametrix in the 0-calculus with bounds]
\label{ThmB}
  Let $\alpha_0<\alpha_1$. Let $P\in\Psi_0^m(X,\Omegazero^{\frac12}X)$ be fully elliptic at all weights $\alpha\in[\alpha_0-\eps,\alpha_1+\eps]$ for some small $\eps>0$. Then there exists operators
  \begin{equation}
  \label{EqBPx}
    Q, Q' \in \Psi_0^{-m}(X,\Omegazero^{\frac12}X) + \Psi_0^{-\infty,(\alpha_1,-\alpha_0+(n-1))}(X,\Omegazero^{\frac12}X) + \Psi^{-\infty,(\alpha_1,-\alpha_0+(n-1))}(X,\Omegazero^{\frac12}X)
  \end{equation}
  so that $P Q=I-R$ and $Q'P=I-R'$ with
  \[
     R \in \Psi^{-\infty,(\infty,-\alpha_0+(n-1))}(X,\Omegazero^{\frac12}X), \qquad
     R' \in \Psi^{-\infty,(\alpha_1,\infty)}(X,\Omegazero^{\frac12}X).
  \]
\end{thm}

\begin{rmk}[Existence of a gap when $X$ is compact]
\label{RmkBCompact}
  When $X$ is compact and $P$ is fully elliptic at the weight $\alpha\in\R$, then the hypotheses of Theorem~\ref{ThmB} are satisfied when $\alpha_0<\alpha$ and $\alpha_1>\alpha$ are sufficiently close to $\alpha$ and $\eps>0$ is sufficiently small.
\end{rmk}

\begin{proof}[Proof of Theorem~\usref{ThmB}]
  Since the arguments are very similar to and indeed more transparent (requiring less bookkeeping) than those in~\S\ref{SsPC}, we shall be brief. We take $Q_0\in\Psi_0^{-m}(X,\Omegazero^{\frac12}X)$ to be a symbolic parametrix, so $P Q_0=I-R_0$ with $R_0\in\Psi_0^{-\infty}(X,\Omegazero^{\frac12}X)$. We can apply Proposition~\ref{PropPNInv} to write the inverse $\hat N(P,y,\hat\eta)^{-1}$ (of $\hat N(P,y,\hat\eta)$ regarded as a map between the spaces in~\eqref{EqPNInv}) for each $y$ separately as the reduced normal operator at $y$ of an element of the large 0-calculus. Note that the index set of $\hat N(P,y,\hat\eta)^{-1}$ at $\lb_\bop$ has real part larger than $\alpha_1+\eps$, and the index set at $\rb_\bop$ has real part larger than $-(\alpha_0-\eps)$. Assembling these inverses, we can thus construct an operator
  \[
    P^-\in\Psi_0^{-m,(\alpha'_1,-\alpha'_0+(n-1))}(X,\Omegazero^{\frac12}X)
  \]
  with $\hat N(P^-,y,\hat\eta)=N(P,y,\hat\eta)^{-1}$ for all $y,\hat\eta$. Here, we denote by $\alpha'_0\in[\alpha_0-\eps,\alpha_0)$ and $\alpha'_1\in(\alpha_1,\alpha_1+\eps]$ two weights which may increase (in the case of $\alpha'_0$) or decrease (in the case of $\alpha'_1$) from line to line throughout the rest of the proof.

  Proposition~\ref{PropCCBounds} gives $Q_1:=P^-R_0\in\Psi_0^{-\infty,(\alpha'_1,-\alpha'_0+(n-1))}(X,\Omegazero^{\frac12}X)$, and we have
  \[
    R_1=R_0-P Q_1\in\rho_\ff\Psi_0^{-\infty,(\alpha'_1,-\alpha'_0+(n-1))}(X,\Omegazero^{\frac12}X).
  \]
  At $\lb$, we can solve this error away using indicial operator arguments, thus producing an operator $Q_2\in\rho_\ff\Psi_0^{-\infty,(\alpha'_1,\infty)}(X,\Omegazero^{\frac12}X)$ so that
  \[
    R_2=R_1-P Q_2 \in \rho_\ff\Psi_0^{-\infty,(\infty,-\alpha'_0+(n-1))}(X,\Omegazero^{\frac12}X).
  \]
  This is solved away using an asymptotic Neumann series, i.e.\ taking an operator $\tilde R_2\in\rho_\ff\Psi_0^{-\infty,(\infty,-\alpha'_0+(n-1))}(X,\Omegazero^{\frac12}X)$ with $\tilde R_2\sim\sum_{j=1}^\infty R_2^j$.\footnote{The composition properties of elements of $\rho_\ff^j\Psi_0^{-\infty,(\beta_\lb,\beta_\rb)}(X,\Omegazero^{\frac12}X)$ can be deduced for general $j$ from those for $j=0$ by writing $\rho_\ff=x'\rho_\rb^{-1}$.} Defining $Q=(Q_0+Q_1+Q_2)(I+\tilde R_2)$, we then have
  \[
    P Q = I-R,\qquad R \in \Psi_0^{-\infty,(\infty,\infty,-\alpha'_0+(n-1))}(X,\Omegazero^{\frac12}X) \subset \Psi^{-\infty,(\infty,-\alpha'_0+(n-1))}(X,\Omegazero^{\frac12}X),
  \]
  as desired. In the membership~\eqref{EqBPx}, we use that
  \[
    \Psi_0^{-\infty,(\alpha_1,\alpha_1-\alpha_0+(n-1),-\alpha_0+(n-1))}(X,\Omegazero^{\frac12}X)=\Psi^{-\infty,(\alpha_1,-\alpha_0+(n-1))}(X,\Omegazero^{\frac12}X).
  \]

  We can construct a left parametrix $Q'$ as usual as the adjoint of a right parametrix $(Q')^*$ for $P^*$. The proof is complete.
\end{proof}

\begin{cor}[Generalized inverse in the 0-calculus with bounds]
\label{CorBGen}
  Suppose $X$ is compact, and let $P$ be an operator satisfying the assumptions of Theorem~\usref{ThmB} (see also Remark~\usref{RmkBCompact}). Then for any $\alpha\in[\alpha_0,\alpha_1]$, the operator $P$ is a Fredholm operator as in~\eqref{EqPGenP}. Put
  \[
    \beta_\lb:=\min(\alpha_1,-\alpha_0+(n-1)+2\alpha)>\alpha,\qquad
    \beta_\rb:=\beta_\lb-2\alpha+(n-1)>-\alpha+(n-1).
  \]
  Then the generalized inverse $G$ of $P$ satisfies
  \[
    G \in \Psi_0^{-m}(X,\Omegazero^{\frac12}X) + \Psi_0^{-\infty,(\beta_\lb,\beta_\rb)}(X,\Omegazero^{\frac12}X) + \Psi^{-\infty,(\beta_\lb,\beta_\rb)}(X,\Omegazero^{\frac12}X).
  \]
\end{cor}
\begin{proof}
  We leave the proof to the reader, as it is a straightforward modification of that of Corollary~\ref{CorPGen}.
\end{proof}

\bibliographystyle{alphaurl}

\end{document}